\documentclass[12pt,english]{article}
\usepackage[T1]{fontenc}
\usepackage[latin9]{inputenc}
\usepackage{geometry}
\usepackage{babel}
\usepackage{array}
\usepackage{longtable}
\usepackage{mathtools}
\usepackage{amsmath}
\usepackage{amsthm}
\usepackage{amssymb}
\usepackage[unicode=true]
 {hyperref}
\usepackage{breakurl}

\makeatletter

\providecommand{\tabularnewline}{\\}

\numberwithin{equation}{section}
\numberwithin{figure}{section}
\theoremstyle{plain}
\newtheorem{thm}{\protect\theoremname}[section]
  \theoremstyle{plain}
  \newtheorem{lem}[thm]{\protect\lemmaname}
  \theoremstyle{plain}
  \newtheorem{cor}[thm]{\protect\corollaryname}

\usepackage{bm}

\usepackage{tikz}
\usetikzlibrary{arrows}

\usepackage{colortbl}
\usepackage{xparse}
\usepackage{caption}

\DeclareMathOperator{\Des}{Des}
\DeclareMathOperator{\pk}{pk}
\DeclareMathOperator{\lpk}{lpk}
\DeclareMathOperator{\Comp}{Comp}
\DeclareMathOperator{\br}{br}
\DeclareMathOperator{\udr}{udr}
\DeclareMathOperator{\des}{des}
\DeclareMathOperator{\altdes}{altdes}
\DeclareMathOperator{\st}{st}
\DeclareMathOperator{\inv}{inv}
\DeclareMathOperator{\val}{val}
\DeclareMathOperator{\Exp}{Exp}
\DeclareMathOperator{\maj}{maj}
\DeclareMathOperator{\imaj}{imaj}
\DeclareMathOperator{\fdes}{fdes}
\DeclareMathOperator{\Orb}{Orb}
\DeclareMathOperator{\Av}{Av}
\DeclareMathOperator{\dasc}{dasc}
\DeclareMathOperator{\ddes}{ddes}
\DeclareMathOperator{\hk}{hk}
\DeclareMathOperator{\nlc}{nlc}
\DeclareMathOperator{\tc}{tc}

\usepackage{titlesec}
\titleformat{\section}
  {\normalfont\large\bfseries}{\thesection.}{.8ex}{} 
\titleformat{\subsection}
  {\normalfont\bfseries}{\thesubsection.}{.8ex}{} 

\usetikzlibrary{shapes}

\tikzstyle{pathdefault}=[draw, line width=1, solid, color=black]
\tikzstyle{nodedefault}=[circle, inner sep=1.5, fill=black]
\tikzstyle{empty}=[]
\tikzstyle{nodeellipsis}=[circle, inner sep=0.5, fill=black]
\tikzstyle{pathcolor1}=[draw, line width=1.3, densely dashed, color=red]
\tikzstyle{pathcolor2}=[draw, line width=1.6, densely dotted, color=blue]
\tikzstyle{pathcolorlight}=[draw, line width=1, dotted, color=lightgray]

\tikzstyle{arbpathcolor0}=[line width=1, dashdotted, color=black]
\tikzstyle{arbpathcolor1}=[line width=1, densely dashed, color=red]
\tikzstyle{arbpathdefault}=[line width=1, densely dotted, color=blue]

\newcounter{id}
\newcommand{\drawlinedotswithstyle}[4]{
 \def\x{{#3}}
 \def\y{{#4}}
 \tikzstyle{thispathstyle}=[#1]
 \tikzstyle{thisnodestyle}=[#2]
 \setcounter{id}{-1} 
 \foreach \j in {#3}{\stepcounter{id}} 
 \foreach \i in {1,...,\the\value{id}}{  
  \path[thispathstyle] (\x[\i],\y[\i]) --(\x[\i-1],\y[\i-1]); 
 }
 \foreach \i in {1,...,\the\value{id}}{  
  \node[thisnodestyle] at (\x[\i],\y[\i]) {}; 
 }
 \node[thisnodestyle] at (\x[0],\y[0]) {}; 
}

\DeclareDocumentCommand{\drawlinedots}{ O{pathdefault} O{nodedefault} m m}{\drawlinedotswithstyle{#1}{#2}{#3}{#4}}

\let\originalleft\left
\let\originalright\right
\renewcommand{\left}{\mathopen{}\mathclose\bgroup\originalleft}
\renewcommand{\right}{\aftergroup\egroup\originalright}

\makeatother

  \providecommand{\corollaryname}{Corollary}
  \providecommand{\lemmaname}{Lemma}
\providecommand{\theoremname}{Theorem}

\begin{document}
\global\long\def\Des{\operatorname{Des}}

\global\long\def\pk{\operatorname{pk}}

\global\long\def\lpk{\operatorname{lpk}}

\global\long\def\Comp{\operatorname{Comp}}

\global\long\def\br{\operatorname{br}}

\global\long\def\udr{\operatorname{udr}}

\global\long\def\des{\operatorname{des}}

\global\long\def\maj{\operatorname{maj}}

\global\long\def\imaj{\operatorname{imaj}}

\global\long\def\altdes{\operatorname{altdes}}

\global\long\def\st{\operatorname{st}}

\global\long\def\Orb{\operatorname{Orb}}

\global\long\def\inv{\operatorname{inv}}

\global\long\def\val{\operatorname{val}}

\global\long\def\fdes{\operatorname{fdes}}

\global\long\def\Exp{\operatorname{Exp}}

\global\long\def\neg{\operatorname{neg}}

\global\long\def\Av{\operatorname{Av}}

\global\long\def\dasc{\operatorname{dasc}}

\global\long\def\ddes{\operatorname{ddes}}

\global\long\def\hk{\operatorname{hk}}

\global\long\def\nlc{\operatorname{nlc}}

\global\long\def\tc{\operatorname{tc}}

\title{Eulerian polynomials and descent statistics}

\author{Yan Zhuang\\
Department of Mathematics\\
Brandeis University\thanks{MS 050, Waltham, MA 02453}\\
\texttt{\href{mailto:zhuangy@brandeis.edu}{zhuangy@brandeis.edu}}}
\maketitle
\begin{abstract}
We prove several identities expressing polynomials counting permutations
by various descent statistics in terms of Eulerian polynomials, extending
results of Stembridge, Petersen, and Br\"and\'en. Additionally,
we find $q$-exponential generating functions for $q$-analogues of
these descent statistic polynomials that also keep track of the inversion
number or inverse major index. We also present identities relating
several of these descent statistic polynomials to refinements of type
B Eulerian polynomials and flag descent polynomials by the number
of negative letters of a signed permutation. Our methods include permutation
enumeration techniques involving noncommutative symmetric functions,
the modified Foata--Strehl action, and a group action of Petersen
on signed permutations. Notably, the modified Foata--Strehl action
yields an analogous relation between Narayana polynomials and the
joint distribution of the peak number and descent number over 231-avoiding
permutations, which we also interpret in terms of binary trees and
Dyck paths.
\end{abstract}
\textbf{\small{}Keywords:}{\small{} permutations, signed permutations,
Eulerian polynomials, descent statistics, noncommutative symmetric
functions, modified Foata--Strehl action}{\let\thefootnote\relax}

\tableofcontents{}

\section{Introduction}

Let $\pi=\pi_{1}\pi_{2}\cdots\pi_{n}$ be a permutation in $\mathfrak{S}_{n}$,
the set of permutations of $[n]=\{1,2,\dots,n\}$, which are called
$n$\textit{-permutations}.\footnote{By convention, we take $\mathfrak{S}_{0}$ to consist of only the
empty word.} Also, let $\left|\pi\right|$ be the length of $\pi$---so that
$\left|\pi\right|=n$ whenever $\pi\in\mathfrak{S}_{n}$---and let
$\mathfrak{S}\coloneqq\bigcup_{n=0}^{\infty}\mathfrak{S}_{n}$. We
say that $i\in[n-1]$ is a \textit{descent} of an $n$-permutation
$\pi$ if $\pi_{i}>\pi_{i+1}$. Every permutation can be uniquely
decomposed into a sequence of maximal increasing consecutive subsequences---or
equivalently, maximal consecutive subsequences containing no descents---which
we call \textit{increasing runs}. For example, the descents of $\pi=85712643$
are $1$, $3$, $6$, and $7$, and the increasing runs of $\pi$
are $8$, $57$, $126$, $4$, and $3$.

Let $\des(\pi)$ denote the number of descents of $\pi$. Then it
is clear that the number of increasing runs of $\pi$ is $\des(\pi)+1$
when $\left|\pi\right|\geq1$, i.e., when $\pi$ is nonempty. The
polynomial 
\[
A_{n}(t)\coloneqq\sum_{\pi\in\mathfrak{S}_{n}}t^{\des(\pi)+1}
\]
for $n\geq1$ is called the $n$th \textit{Eulerian polynomial}. We
set $A_{0}(t)=1$ by convention.\footnote{In general, for all polynomials defined in this paper counting permutations
by various statistics---such as $P_{n}^{\pk}(t)$, $P_{n}^{(\pk,\des)}(y,t)$,
etc.---we set the $0$th polynomial to be 1 by convention.} The exponential generating function for Eulerian polynomials is well
known: 
\[
\sum_{n=0}^{\infty}A_{n}(t)\frac{x^{n}}{n!}=\frac{1-t}{1-te^{(1-t)x}}.
\]
The Eulerian polynomials have a rich history and appear in many contexts
in combinatorics; see \cite{Petersen2015} for a detailed exposition.

The Eulerian polynomials are closely related to the distribution of
other descent statistics: permutation statistics that depend only
on the descent set and length of a permutation. Specifically, it is
known that polynomials counting permutations by various descent statistics---the
number of peaks, left peaks, and biruns---can be expressed in terms
of Eulerian polynomials. 
\begin{itemize}
\item We say that $i$ (where $2\leq i\leq n-1$) is a \textit{peak} of
$\pi=\pi_{1}\pi_{2}\cdots\pi_{n}$ if $\pi_{i-1}<\pi_{i}>\pi_{i+1}$,
and let $\pk(\pi)$ be the number of peaks of $\pi$. For example,
the peaks of $\pi=85712643$ are $3$ and $6$, and so $\pk(\pi)=2$.
The peak polynomials 
\[
P_{n}^{\pk}(t)\coloneqq\sum_{\pi\in\mathfrak{S}_{n}}t^{\pk(\pi)+1}
\]
are related to the Eulerian polynomials by the identity 
\begin{equation}
A_{n}(t)=\left(\frac{1+t}{2}\right)^{n+1}P_{n}^{\pk}\left(\frac{4t}{(1+t)^{2}}\right)\label{e-apk}
\end{equation}
for $n\geq1$, which was first stated explicitly by Stembridge \cite{Stembridge1997}
as a result of his theory of enriched $P$-partitions, but also follows
from an earlier construction of Shapiro, Woan, and Getu \cite{Shapiro1983}
which was later rediscovered by Br\"and\'en as the ``modified Foata--Strehl
action'' \cite{Braenden2008}, a variant of a group action on permutations
originally defined by Foata and Strehl \cite{Foata1974}. Making a
substitution yields the equivalent identity 
\[
P_{n}^{\pk}(t)=\left(\frac{2}{1+v}\right)^{n+1}A_{n}(v)
\]
where $v=\frac{2}{t}(1-\sqrt{1-t})-1$.
\item We say that $i\in[n-1]$ is a \textit{left peak} of $\pi\in\mathfrak{S}_{n}$
if either $i$ is a peak, or if $i=1$ and 1 is a descent.\footnote{Equivalently, $i$ is a left peak of $\pi$ if it is a peak of the
permutation $0\pi$ obtained by prepending a letter 0 to $\pi$. We
can also define a ``right peak'' in the analogous way, but it is easy
to see from symmetry that the number of left peaks and the number
of right peaks are equidistributed over $\mathfrak{S}_{n}$.} The number of left peaks of $\pi$ is denoted $\lpk(\pi)$. For example,
the left peaks of $\pi=85712643$ are $1$, $3$ and $6$, and so
$\lpk(\pi)=3$. Let 
\[
P_{n}^{\lpk}(t)\coloneqq\sum_{\pi\in\mathfrak{S}_{n}}t^{\lpk(\pi)}
\]
be the polynomial counting $n$-permutations by left peaks. Using
a modification of enriched $P$-partitions called ``left enriched
$P$-partitions'', Petersen \cite[Observation 3.1.2]{Petersen2006}
proved the identity 
\begin{equation}
\sum_{k=0}^{n}{n \choose k}2^{k}(1-t)^{n-k}A_{k}(t)=(1+t)^{n}P_{n}^{\lpk}\left(\frac{4t}{(1+t)^{2}}\right)\label{e-lpk}
\end{equation}
for all $n$. Equivalently,
\begin{align*}
P_{n}^{\lpk}(t) & =\frac{1}{(1+v)^{n}}\sum_{k=0}^{n}{n \choose k}2^{k}(1-v)^{n-k}A_{k}(v)
\end{align*}
where again $v=\frac{2}{t}(1-\sqrt{1-t})-1$.
\item A \textit{birun}\footnote{Biruns are more commonly called \textit{alternating runs}, but since
the term ``alternating run'' is used for a different concept in the
related papers \cite{Gessel2014,Zhuang2016}, we use the term ``birun''
which was suggested by Stanley \cite[Section 4]{Stanley2008}.} of a permutation is a maximal monotone consecutive subsequence---that
is, an increasing run of length at least two or a decreasing run of
length at least two---and the number of biruns of $\pi$ is denoted
$\br(\pi)$. For example, the biruns of $\pi=85712643$ are $85$,
$57$, $71$, $126$ and $643$, and thus $\br(\pi)=5$. Define 
\[
P_{n}^{\br}(t)\coloneqq\sum_{\pi\in\mathfrak{S}_{n}}t^{\br(\pi)}.
\]
Using differential equations, David and Barton \cite{David1962} proved
the identity 
\[
P_{n}^{\br}(t)=\left(\frac{1+t}{2}\right)^{n-1}(1+v)^{n+1}A_{n}\left(\frac{1-v}{1+v}\right)
\]
for $n\geq1$, where $v=\sqrt{\frac{1-t}{1+t}}$.
\end{itemize}
In this paper, we establish several new identities which similarly
express polynomials counting permutations by certain descent statistics
in terms of Eulerian polynomials, including refinements of the known
results on $\pk$ and $\lpk$ proved by Stembridge and Petersen, respectively.
Futhermore, we find expressions for $q$-exponential generating functions
for $q$-analogues of these descent statistic polynomials that also
keep track of the inversion number (or inverse major index; see Subsection
4.4), although there are no analogous expressions in terms of the
$q$-Eulerian polynomials $A_{n}(q,t)=\sum_{\pi\in\mathfrak{S}_{n}}q^{\inv(\pi)}t^{\des(\pi)+1}$.
In particular, the descent statistics that we consider are the ordered
pairs $(\pk,\des)$ and $(\lpk,\des)$, the number of ``up-down runs''
$\udr$, and the triple $(\lpk,\val,\des)$ where $\val$ is the number
of ``valleys''.\footnote{We shall see that $(\lpk,\val,\des)$ is equivalent to $(\udr,\des)$,
but the former is easier to work with due to technical constraints
surrounding the latter.}

The idea of descents has been extended to other finite Coxeter groups,
the most important of which (from the perspective of permutation enumeration)
are the hyperoctahedral groups $\mathfrak{B}_{n}$, consisting of
``signed permutations''. In particular, two notions of descent number
for signed permutations are the ``type B descent number'' $\des_{B}$
and the ``flag descent number'' $\fdes$, which give rise to analogues
of Eulerian polynomials: type B Eulerian polynomials and flag descent
polynomials. We refine these polynomials to also keep track of the
number of ``negative letters'' $\neg$ of a signed permutation, and
uncover similar identities that relate the distribution of $(\lpk,\des)$
over $\mathfrak{S}_{n}$ to that of $(\neg,\des_{B})$ over $\mathfrak{B}_{n}$---which
specializes to a connection between $\lpk$ and $\des_{B}$ previously
discovered by Petersen---and relate the distribution of $(\lpk,\val,\des)$
over $\mathfrak{S}_{n}$ to that of $(\neg,\fdes)$ over $\mathfrak{B}_{n}$,
which specializes to a previously unknown connection between $\udr$
and $\fdes$.

The methods that we employ are twofold. Our main results are obtained
using general permutation enumeration techniques involving noncommutative
symmetric functions developed by Gessel \cite{gessel-thesis} and
later extended by Gessel and the present author \cite{Gessel2014,Zhuang2016}.
Later, we use two group actions---the modified Foata--Strehl action
and an action of Petersen on signed permutations---to prove generalizations
of some of our basic results for $(\pk,\des)$, $(\lpk,\des)$, and
$(\lpk,\val,\des)$. Notably, the generalized result for $(\pk,\des)$
specializes to a result relating the Narayana polynomials and the
$(\pk,\des)$ polynomials for 231-avoiding permutations, which we
also interpret in terms of binary trees and Dyck paths.

The organization of this paper is as follows. In Section 2, we review
basic definitions from permutation enumeration and some classical
results on counting permutations with a prescribed descent set, establish
some preliminary facts about the aforementioned refinements of type
B Eulerian polynomials and flag descent polynomials, and prove several
identities relating the Eulerian polynomials, refined type B Eulerian
polynomials, and refined flag descent polynomials. In Section 3, we
introduce parts of the theory of noncommutative symmetric functions
relevant to this work. In Section 4, we prove several identities involving
noncommutative symmetric functions and use them to obtain our main
results. Finally, in Section 5, we introduce the group actions of
Br\"and\'en and Petersen and use them to generalize some of our
results from Section 4. 

A summary of every statistic appearing in this paper is given in the
appendix.

\section{Permutations and descents}

\subsection{Descent sets, compositions, and statistics}

We begin by reviewing some basic material from permutation enumeration
relating to descents. Recall that a descent of an $n$-permutation
$\pi$ is an index $i\in[n-1]$ such that $\pi_{i}>\pi_{i+1}$. The
set of descents, or \textit{descent set}, of a permutation $\pi$
is denoted $\Des(\pi)$. We observed that the descents of a permutation
separate it into increasing runs. In fact, the lengths of the increasing
runs of a permutation determine its descents, and vice versa. Sometimes
it is actually more convenient to represent a descent set of an $n$-permutation
with a composition of $n$ which encodes the lengths of its increasing
runs.

Given a subset $S\subseteq[n-1]$ with elements $s_{1}<s_{2}<\cdots<s_{j}$,
let $\Comp(S)$ be the composition $(s_{1},s_{2}-s_{1},\dots,s_{j}-s_{j-1},n-s_{j})$
of $n$, and given a composition $L=(L_{1},L_{2},\dots,L_{k})$, let
$\Des(L)\coloneqq\{L_{1},L_{1}+L_{2},\dots,L_{1}+\cdots+L_{k-1}\}$
be the corresponding subset of $[n-1]$. Then, $\Comp$ and $\Des$
are inverse bijections. If $\pi$ is an $n$-permutation with descent
set $S\subseteq[n-1]$, then we call $\Comp(S)$ the \textit{descent
composition} of $\pi$, which we also denote by $\Comp(\pi)$. Note
that the descent composition of $\pi$ gives the lengths of the increasing
runs of $\pi$. Conversely, if $\pi$ has descent composition $L$,
then its descent set $\Des(\pi)$ is $\Des(L)$.

We partially order compositions of $n$ by reverse refinement, that
is, $L=(L_{1},\dots,L_{k})$ covers $M$ if and only if $M$ can be
obtained from $L$ by replacing two consecutive parts $L_{i}$ and
$L_{i+1}$ with $L_{i}+L_{i+1}$. For example, we have $(7,6)<(1,2,4,5,1)$.
Note that if $L$ and $M$ are descent compositions, then $L\leq M$
under this ordering if and only if $\Des(L)\subseteq\Des(M)$; in
other words, $\Comp$ and $\Des$ are order-preserving bijections.

A permutation statistic $\st$ is called a \textit{descent statistic}
if it depends only on the descent composition, that is, if $\Comp(\pi)=\Comp(\sigma)$
implies $\st(\pi)=\st(\sigma)$ for any two $\pi,\sigma\in\mathfrak{S}$.
Equivalently, $\st$ is a descent statistic if it depends only on
the descent set and length of a permutation. Examples of descent statistics
include all of the statistics discussed in the introduction: the descent
number $\des$, the peak number $\pk$, the left peak number $\lpk$,
and the number of biruns $\br$. Ordered tuples of descent statistics
such as $(\pk,\des)$ and $(\lpk,\des)$ are also descent statistics.

Two further examples of descent statistics are the number of ``valleys''
and the number of ``up-down runs'', defined as follows. We say that
$i$ (where $2\leq i\leq n-1$) is a \textit{valley} of $\pi$ if
$\pi_{i-1}>\pi_{i}<\pi_{i+1}$, and the number of valleys of $\pi$
is denoted $\val(\pi)$. The valleys of $\pi=85712643$ are $2$ and
$4$, so $\val(\pi)=2$.

An \textit{up-down run} of a permutation $\pi=\pi_{1}\pi_{2}\cdots\pi_{n}$
is either a birun of $\pi$ or the letter $\pi_{1}$ if it is an increasing
run of length 1, and the number of up-down runs of $\pi$ is denoted
$\udr(\pi)$. The up-down runs of $\pi=85712643$ are $8$, $85$,
$57$, $71$, $126$, and $643$, so $\udr(\pi)=6$.

The peak number $\pk$ and valley number $\val$---as well as the
ordered pairs $(\pk,\des)$ and $(\val,\des)$---are equidistributed
on $\mathfrak{S}_{n}$ due to symmetry, so we will not consider the
valley number per se. However, the valley number and left peak number
are closely related to the number of up-down runs, which we establish
in the next lemma.
\begin{lem}
\label{l-udr} Let $\pi\in\mathfrak{S}_{n}$ with $n\geq1$. Then:

\begin{enumerate}
\item [\normalfont{(a)}] $\udr(\pi)=\lpk(\pi)+\val(\pi)+1$
\item [\normalfont{(b)}] $\lpk(\pi)=\left\lfloor \udr(\pi)/2\right\rfloor $
\item [\normalfont{(c)}] $\val(\pi)=\left\lfloor (\udr(\pi)-1)/2\right\rfloor $
\item [\normalfont{(d)}] If $n-1$ is a descent of $\pi$, then $\lpk(\pi)=\val(\pi)+1$.
Otherwise, $\lpk(\pi)=\val(\pi)$.
\end{enumerate}
\end{lem}
\begin{proof}
Every up-down run except the final one ends with either a left peak
or a valley, and in fact these up-down runs alternate between ending
with a left peak and ending with a valley, beginning with a left peak.
For example, if $\udr(\pi)=5$, then the first up-down run ends with
a left peak, the second ends with a valley, the third ends with a
left peak, and the fourth ends with a valley. It is clear that this
accounts for every left peak and every valley, which proves (a). Now,
note that either $\lpk(\pi)=\val(\pi)+1$ or $\lpk(\pi)=\val(\pi)$;
this depends completely on whether the penultimate up-down run ends
with a left peak or a valley, which is determined by whether the final
up-down run is increasing or decreasing (i.e., whether the final run
is long or short); this proves (d). Finally, (b) and (c) follow from
(a) and (d).
\end{proof}
Lemma \ref{l-udr} shows that not only does $(\lpk,\val)$ determines
$\udr$, but $\udr$ also determines $(\lpk,\val)$. In other words,
$\udr$ and $(\lpk,\val)$ are equivalent permutation statistics.

An example of a permutation statistic that is not a descent statistic
is the ``inversion number''. An \textit{inversion} in an $n$-permutation
is a pair of indices $(i,j)$ with $1\leq i<j\leq n$ such that $\pi_{i}>\pi_{j}$.
Then the number of inversions of $\pi$ is denoted $\inv(\pi)$. For
example, the inversions of $\pi=1432$ are $(4,3)$, $(4,2)$, and
$(3,2)$, so $\inv(\pi)=3$. It is well known that the polynomial
counting $n$-permutations by inversion number is given by the $n$th
$q$-factorial 
\[
[n]_{q}!\coloneqq(1+q)(1+q+q^{2})\cdots(1+q+\cdots+q^{n-1}),
\]
i.e., 
\[
\sum_{\pi\in\mathfrak{S}_{n}}q^{\inv(\pi)}=[n]_{q}!.
\]

We give two key remarks before continuing. First, the definitions
and properties of descents, increasing runs, descent compositions,
and descent statistics extend naturally to words on any totally ordered
alphabet such as $[n]$ or $\mathbb{P}$ (the positive integers) if
we replace the strict inequality $<$ with the weak inequality $\leq$,
which reflects the fact that increasing runs are allowed to be weakly
increasing in this setting. For example, $i$ is a peak of the word
$w=w_{1}w_{2}\cdots w_{n}$ if $w_{i-1}\leq w_{i}>w_{i+1}$. 

Finally, recall that by definition, two permutations (or words) with
the same descent composition must have the same value of $\st$ if
$\st$ is a descent statistic. Hence, we shall use the notation $\st(L)$
to indicate the value of a descent statistic $\st$ on any permutation
(or word) with descent composition $L$.

\subsection{Counting permutations with a prescribed descent set}

If $L=(L_{1},\dots,L_{k})$ is a composition of $n$, then we let
$l(L)$ denote the number of parts of $L$ and we write ${n \choose L}$
for the multinomial coefficient ${n \choose L_{1},\dots,L_{k}}$.
Similarly, we write ${n \choose L}_{\!q}$ for the $q$-multinomial
coefficient 
\[
{n \choose L_{1},\dots,L_{k}}_{\!\!q}\coloneqq\frac{[n]_{q}!}{[L_{1}]_{q}!\,[L_{2}]_{q}!\cdots[L_{k}]_{q}!}.
\]

\begin{lem}
\label{l-Descont} Let $L$ be a composition of $n$. Then:

\begin{enumerate}
\item [\normalfont{(a)}] The number of $n$-permutations with descent composition
$K\leq L$---or equivalently, with descent set contained in $\Des(L)$---is
the multinomial coefficient ${n \choose L}$.
\item [\normalfont{(b)}] The polynomial counting $n$-permutations with
descent composition $K\leq L$---or equivalently, with descent set
contained in $\Des(L)$---by inversion number is the $q$-multinomial
coefficient ${n \choose L}_{\!q}$. That is, 
\[
\sum_{\substack{\pi\in\mathfrak{S}_{n}\\
\Comp(\pi)\leq L
}
}q^{\inv(\pi)}={n \choose L}_{\!\!q}.
\]
\end{enumerate}
\end{lem}
See \cite[Examples 2.2.4 and 2.2.5]{Stanley2011} for proofs. This
result on counting $n$-permutations with a descent set contained
in a prescribed set can then be used to count those with a prescribed
descent set.
\begin{lem}
\label{l-Despre} Let $L$ be a composition of $n$. Then:

\begin{enumerate}
\item [\normalfont{(a)}] The number $\beta(L)$ of $n$-permutations with
descent composition $L$---or equivalently, with descent set $\Des(L)$---is
given by the formula 
\begin{equation}
\beta(L)=\sum_{K\leq L}(-1)^{l(L)-l(K)}{n \choose K}.\label{e-betaL}
\end{equation}
\item [\normalfont{(b)}] The polynomial 
\[
\beta_{q}(L)\coloneqq\sum_{\substack{\pi\in\mathfrak{S}_{n}\\
\Comp(\pi)=L
}
}q^{\inv(\pi)}
\]
counting $n$-permutations with descent composition $L$---or equivalently,
with descent set $\Des(L)$---by inversion number is given by the
formula 
\begin{equation}
\beta_{q}(L)=\sum_{K\leq L}(-1)^{l(L)-l(K)}{n \choose K}_{\!\!q}.\label{e-betaqL}
\end{equation}
\end{enumerate}
\end{lem}
The proof of Lemma \ref{l-Despre} is immediate from Lemma \ref{l-Descont}
and the inclusion-exclusion principle. Part (a) of Lemmas \ref{l-Descont}
and \ref{l-Despre} were originally due to MacMahon \cite{macmahon},
whereas part (b) of these lemmas were due to Stanley \cite{Stanley1976}.

\subsection{Type B permutations, descents, and Eulerian polynomials}

Let $\mathfrak{B}_{n}$ be the set of permutations $\pi=\pi_{-n}\cdots\pi_{-1}\pi_{0}\pi_{1}\cdots\pi_{n}$
of $\{-n,\dots,-1,0,1,\dots,n\}$ satisfying $\pi_{-i}=-\pi_{i}$
for all $-n\leq i\leq n$; we call these \textit{signed $n$-permutations}
(or \textit{type B $n$-permutations}). Let $\mathfrak{\mathfrak{B}}\coloneqq\bigcup_{n=0}^{\infty}\mathfrak{B}_{n}$.
For any signed $n$-permutation $\pi$, we must have $\pi_{0}=0$
and $\pi$ is completely determined by $\{\pi_{1},\dots,\pi_{n}\}$,
so we can write $\pi$ as $\pi=\pi_{1}\cdots\pi_{n}$ with the understanding
that $\pi_{0}=0$ and $\pi_{-i}=-\pi_{i}$ for all $i$. In this way,
we can think of $\mathfrak{S}_{n}$ as the subset of signed permutations
in $\mathfrak{B}_{n}$ with no negative letters among $\{\pi_{1},\dots,\pi_{n}\}$.

For cleaner notation, let us write $\bar{i}$ rather than $-i$ when
writing out the letters of a signed permutation. For example, if $\pi=\pi_{1}\pi_{2}\pi_{3}$
with $\pi_{1}=3$, $\pi_{2}=-2$, and $\pi_{3}=-1$, then we write
$\pi=3\bar{2}\bar{1}$.

We say that $i\in\{0\}\cup[n-1]$ is a \textit{descent} (or \textit{type
B descent}) of $\pi\in\mathfrak{B}_{n}$ if $\pi_{i}>\pi_{i+1}$.
Note that we allow 0 to be a descent, which happens precisely when
$\pi_{1}$ is negative. There are two notions of descent number for
signed permutations that we consider. The \textit{descent number}
(or \textit{type B descent number}) $\des_{B}(\pi)$ is simply the
number of descents of $\pi\in\mathfrak{B}_{n}$, whereas the \textit{flag
descent number} $\fdes(\pi)$ is defined by 
\[
\fdes(\pi)\coloneqq\begin{cases}
2\des_{B}(\pi), & \mbox{if }\pi_{1}>0\\
2\des_{B}(\pi)-1, & \mbox{if }\pi_{1}<0;
\end{cases}
\]
that is, every descent except 0 is counted twice. For example, let
$\pi=\bar{4}72\bar{6}\bar{3}51$. Then the descents of $\pi$ are
0, 2, 3, and 6, so $\des_{B}(\pi)=4$ and $\fdes(\pi)=7$.

We define 
\[
B_{n}(t)\coloneqq\sum_{\pi\in\mathfrak{B}_{n}}t^{\des_{B}(\pi)}
\]
and 
\[
F_{n}(t)\coloneqq\sum_{\pi\in\mathfrak{B}_{n}}t^{\fdes(\pi)},
\]
which are type B analogues of Eulerian polynomials using the descent
number and flag descent number, respectively. We call $B_{n}(t)$
the $n$th \textit{type B Eulerian polynomial} and $F_{n}(t)$ the
$n$th \textit{flag descent polynomial}.

The exponential generating function 
\begin{equation}
\sum_{n=0}^{\infty}\frac{B_{n}(t)}{(1-t)^{n+1}}\frac{x^{n}}{n!}=\frac{e^{x}}{1-te^{2x}}\label{e-bn}
\end{equation}
was found by Steingr\'imsson \cite{Steingrimsson1994}, and the analogous
formula for the flag descent polynomials 
\begin{equation}
\sum_{n=0}^{\infty}\frac{F_{n}(t)}{(1-t)(1-t^{2})^{n}}\frac{x^{n}}{n!}=\frac{e^{x}}{1-te^{x}}\label{e-fn}
\end{equation}
directly follows from a result of Adin, Brenti, and Roichman \cite[Theorem 4.2]{Adin2001}.

We consider another statistic on $\mathfrak{B}$: the number of negative
letters 
\[
\neg(\pi)\coloneqq\#\{\pi_{i}\mid\pi_{i}<0\mbox{ and }i\in[n]\}.
\]
So given $\pi=\bar{4}72\bar{6}\bar{3}51$, we have $\neg(\pi)=3$.
We refine the polynomials $B_{n}(t)$ and $F_{n}(t)$ by this statistic,
defining 
\[
B_{n}(y,t)\coloneqq\sum_{\pi\in\mathfrak{B}_{n}}y^{\neg(\pi)}t^{\des_{B}(\pi)}
\]
and 
\[
F_{n}(y,t)\coloneqq\sum_{\pi\in\mathfrak{B}_{n}}y^{\neg(\pi)}t^{\fdes(\pi)}.
\]
Later, we will relate the polynomials $B_{n}(y,t)$ to the joint distribution
of $\lpk$ and $\des$ over $\mathfrak{S}_{n}$, and similarly $F_{n}(y,t)$
with the joint distribution of $\lpk$, $\val$, and $\des$. In doing
so, we shall need the following exponential generating functions for
these polynomials. 
\begin{thm}
\label{t-bnegf} 
\[
\sum_{n=0}^{\infty}\frac{B_{n}(y,t)}{(1-t)^{n+1}}\frac{x^{n}}{n!}=\frac{e^{x}}{1-te^{(1+y)x}}
\]
\end{thm}
Note that setting $y=1$ yields Steingr\'imsson's formula (\ref{e-bn}).
\begin{proof}
We begin by proving the identity 
\begin{equation}
\frac{B_{n}(y,t)}{(1-t)^{n+1}}=\sum_{k=0}^{\infty}(ky+(k+1))^{n}t^{k},\label{e-bnbars}
\end{equation}
which was first established by Petersen \cite{Petersen2007} using
a different method.

Consider the left-hand side. Each term in $B_{n}(y,t)$ corresponds
to a signed $n$-permutation with a vertical bar inserted after each
letter corresponding to a descent (and an initial bar if 0 is a descent).
For example, if we have $\pi=\bar{4}72\bar{6}\bar{3}51$, then we
write this as 
\[
|\bar{4}7|2|\bar{6}\bar{3}5|1.
\]
The $1/(1-t)^{n+1}$ factor corresponds to inserting any number of
bars in any of the $n+1$ positions between letters, before the first
letter, or after the final letter. So for example, continuing from
above, we may have 
\[
|\bar{4}7|2||\bar{6}\bar{3}|||5|1|.
\]
Thus the left-hand side of (\ref{e-bnbars}) counts the number of
signed $n$-permutations with any number of bars inserted in any of
the $n+1$ positions and at least one bar in every position corresponding
to a descent, where $y$ is weighting the number of negative letters
and $t$ is weighting the number of bars.

We claim that the right-hand side counts these same barred signed
$n$-permutations. Fix $k\geq0$; this is the number of bars. The
bars create $k+1$ ``boxes'' for inserting letters. For every $i\in[n-1]$,
we make two choices: whether or not to make it negative, and which
box to put it into. The letters in each box are then placed in increasing
order. Note that the first box cannot contain any negative letters;
otherwise, 0 would be a descent, but there would not be a bar preceding
the first letter. Thus, if a letter is made negative, then it contributes
a weight of $y$ and we can place it in any of the $k$ boxes after
the first one. If a letter remains positive, then it can be placed
into any of the $k+1$ boxes. Since there are $n$ letters and the
choices are made independently, we have a total contribution of $(ky+(k+1))^{n}t^{k}$
in the case where there are $k$ bars in total. Summing over all $k$
yields the right-hand side of (\ref{e-bnbars}).

Now, observe that 
\begin{align*}
\sum_{n=0}^{\infty}\frac{B_{n}(y,t)}{(1-t)^{n+1}}\frac{x^{n}}{n!} & =\sum_{n=0}^{\infty}\sum_{k=0}^{\infty}(ky+(k+1))^{n}t^{k}\frac{x^{n}}{n!}\\
 & =\sum_{k=0}^{\infty}e^{(ky+(k+1))x}t^{k}\\
 & =e^{x}\sum_{k=0}^{\infty}(e^{(1+y)x})^{k}t^{k}\\
 & =\frac{e^{x}}{1-te^{(1+y)x}},
\end{align*}
thus completing the proof.
\end{proof}
\begin{thm}
\label{t-fnegf} 
\[
\sum_{n=0}^{\infty}\frac{F_{n}(y,t)}{(1-t)(1-t^{2})^{n}}\frac{x^{n}}{n!}=\frac{e^{x}+te^{(1+y)x}}{1-t^{2}e^{(1+y)x}}
\]
\end{thm}
Setting $y=1$ yields the formula (\ref{e-fn}) of Adin, Brenti, and
Roichman.
\begin{proof}
We first prove the identity 
\begin{equation}
\frac{F_{n}(y,t)}{(1-t)(1-t^{2})^{n}}=\sum_{k=0}^{\infty}(ky+(k+1))^{n}t^{2k}+\sum_{k=0}^{\infty}((k+1)(y+1))^{n}t^{2k+1}.\label{e-fnbars}
\end{equation}
Each term in $F_{n}(y,t)$ corresponds to a signed $n$-permutation
with an arrangement of bars, but now each nonzero descent contributes
a weight of $t^{2}$, which corresponds to two bars. We write each
$\pi\in\mathfrak{B}_{n}$ as $\pi=\pi_{-n}\cdots\pi_{-1}\pi_{1}\cdots\pi_{n}$
(without $\pi_{0}=0$). For each descent $i\in[n-1]$, we insert a
bar immediately after $\pi_{i}$ and a bar immediately before $\pi_{-i}$.
If $i=0$ is a descent, then we insert a single bar between $\pi_{-1}$
and $\pi_{1}$. For example, for $\pi=\bar{4}72\bar{6}\bar{3}51$,
we have 
\[
\bar{1}|\bar{5}36|\bar{2}|\bar{7}4|\bar{4}7|2|\bar{6}\bar{3}5|1.
\]
The $1/(1-t)$ factor corresponds to inserting any number of bars
in the central position (between $\pi_{-1}$ and $\pi_{1}$), and
the $1/(1-t^{2})^{n}$ factor corresponds to inserting any number
of bars in any of the $n$ positions to the right of the central position,
and for each of these bars, a corresponding bar in the position symmetric
about the center. For example, we may have 
\[
|\bar{1}|\bar{5}|||36||\bar{2}|\bar{7}4|\bar{4}7|2||\bar{6}\bar{3}|||5|1|.
\]

We claim that the right-hand side of (\ref{e-fnbars}) counts the
same arrangements. We consider two cases: the number of bars is even
or the number of bars is odd.

\begin{itemize}
\item Suppose that the number of bars is $2k$ for some $k\geq0$. These
bars create $2k+1$ boxes; we will be inserting letters into the right-most
$k+1$ boxes. Again, for each letter, we decide whether or not to
make it negative and decide which box to put it in. If a letter is
made negative, then it contributes a weight of $y$ and it can only
be inserted into the final $k$ boxes; if a letter is not made negative,
then we can insert it into any of the right-most $k+1$ boxes. Order
the letters in each box in increasing order, and for each letter,
insert its negative into the position symmetric about the center.
These are precisely the arrangements that we want; thus we have the
term $(ky+(k+1))^{n}t^{2k}$, and summing over $k$ gives the total
contribution from having an even number of bars.
\item Suppose that the number of bars in $2k+1$ for some $k\geq0$. Then
these bars create $2k+2$ boxes. In this case, both positive and negative
letters can be inserted into any of the right-most $k+1$ boxes; if
a negative letter is inserted in the ($k+2$)nd box, then the central
bar acts as the bar corresponding to the 0 descent. Hence, this contributes
$((k+1)(y+1))^{n}t^{2k}$, and summing over $k$ gives the total contribution
from having an odd number of bars.
\end{itemize}
Now, observe that 
\begin{align*}
\sum_{n=0}^{\infty}\sum_{k=0}^{\infty}(ky+(k+1))^{n}t^{2k}\frac{x^{n}}{n!} & =\sum_{k=0}^{\infty}e^{(k(1+y)+1)x}t^{2k}\\
 & =e^{x}\sum_{k=0}^{\infty}(t^{2}e^{(1+y)x})^{k}\\
 & =\frac{e^{x}}{1-t^{2}e^{(1+y)x}}
\end{align*}
and 
\begin{align*}
\sum_{n=0}^{\infty}\sum_{k=0}^{\infty}((k+1)(y+1))^{n}t^{2k+1}\frac{x^{n}}{n!} & =\sum_{k=0}^{\infty}e^{((k+1)(y+1))x}t^{2k+1}\\
 & =te^{(1+y)x}\sum_{k=0}^{\infty}(t^{2}e^{(1+y)x})^{k}\\
 & =\frac{te^{(1+y)x}}{1-t^{2}e^{(1+y)x}};
\end{align*}
adding these expressions completes the proof.
\end{proof}

\subsection{Several new Eulerian polynomial identities}

Before proceeding, we prove several new identities relating the Eulerian
polynomials $A_{n}(t)$, refined type B Eulerian polynomials $B_{n}(y,t)$,
and refined flag descent polynomials $F_{n}(y,t)$. Unlike the main
results of this paper which will be presented in Section 4, these
results can be obtained simply using the exponential generating functions
established in the previous subsection and will not require the more
sophisticated techniques introduced in the next section.
\begin{thm}
\label{t-bnan} For $n\geq0$, we have
\[
B_{n}(y,t)=\sum_{k=0}^{n}{n \choose k}(1+y)^{k}(1-t)^{n-k}A_{k}(t).
\]
\end{thm}
\begin{proof}
Taking Theorem \ref{t-bnegf}, multiplying both sides by $1-t$, and
then replacing $x$ with $(1-t)x/(1+y)$ yields 
\begin{align*}
\sum_{n=0}^{\infty}\frac{B_{n}(y,t)}{(1+y)^{n}}\frac{x^{n}}{n!} & =\frac{1-t}{1-te^{(1-t)x}}e^{\frac{1-t}{1+y}x}\\
 & =\Big(\sum_{n=0}^{\infty}A_{n}(t)\frac{x^{n}}{n!}\Big)\Big(\sum_{n=0}^{\infty}\left(\frac{1-t}{1+y}\right)^{n}\frac{x^{n}}{n!}\Big)\\
 & =\sum_{n=0}^{\infty}\sum_{k=0}^{n}{n \choose k}\left(\frac{1-t}{1+y}\right)^{n-k}A_{k}(t)\frac{x^{n}}{n!}.
\end{align*}
Equating the coefficients of $x^{n}/n!$ and multiplying both sides
by $(1+y)^{n}$ yields the result.
\end{proof}
By setting $y=1$, we obtain the following corollary.
\begin{cor}
\label{c-bnan} For $n\geq0$, we have
\[
B_{n}(t)=\sum_{k=0}^{n}{n \choose k}2^{k}(1-t)^{n-k}A_{k}(t).
\]
\end{cor}
\begin{thm}
\label{t-fnan} For $n\geq1$, we have
\[
F_{n}(y,t)=\frac{1}{1+t}\Big(\frac{(1+y)^{n}}{t}A_{n}(t^{2})+\sum_{k=0}^{n}{n \choose k}(1+y)^{k}(1-t^{2})^{n-k}A_{k}(t^{2})\Big).
\]
\end{thm}
\begin{proof}
It is readily checked that the statement of Theorem \ref{t-fnegf}
is equivalent to
\[
\frac{1}{1-t}\Big(1+t\sum_{n=1}^{\infty}\frac{F_{n}(y,t)}{(1-t^{2})^{n}}\frac{x^{n}}{n!}\Big)=\frac{1+te^{x}}{1-t^{2}e^{(1+y)x}}.
\]
Multiplying both sides by $1-t^{2}$ and replacing $x$ with $(1-t^{2})x/(1+y)$
yields 
\begin{align*}
(1+t)\Big(1+t\sum_{n=1}^{\infty}\frac{F_{n}(y,t)}{(1+y)^{n}}\frac{x^{n}}{n!}\Big) & =\frac{1-t^{2}}{1-t^{2}e^{(1+t^{2})x}}(1+te^{\frac{1-t^{2}}{1+y}x})\\
 & =\Big(\sum_{n=0}^{\infty}A_{n}(t^{2})\frac{x^{n}}{n!}\Big)\Big(1+t\sum_{n=0}^{\infty}\left(\frac{1-t^{2}}{1+y}\right)^{n}\frac{x^{n}}{n!}\Big)\\
 & =\sum_{n=0}^{\infty}\Big(A_{n}(t^{2})+t\sum_{k=0}^{n}{n \choose k}\left(\frac{1-t^{2}}{1+y}\right)^{n-k}A_{k}(t^{2})\Big)\frac{x^{n}}{n!}.
\end{align*}
Equating the coefficients of $x^{n}/n!$ and dividing both sides by
$t(1+t)/(1+y)^{n}$ yields the result.
\end{proof}
We can set $y=1$ to obtain an identity relating $F_{n}(t)$ and $A_{n}(t)$,
but the nicer identity
\begin{equation}
F_{n}(t)=\frac{(1+t)^{n}}{t}A_{n}(t)\label{e-fnan}
\end{equation}
can be obtained by directly comparing the generating functions of
$F_{n}(t)$ and $A_{n}(t)$, and can also be recovered as a specialization
of a more general identity of Adin, Brenti, and Roichman \cite[Theorem 4.4]{Adin2001}.
\begin{thm}
\label{t-fnbn} For $n\geq1$, we have
\[
F_{n}(y,t)=\frac{1}{1+t}\Big(B_{n}(y,t^{2})+\frac{1}{t}\sum_{k=0}^{n}(-1)^{n-k}{n \choose k}(1-t^{2})^{n-k}B_{k}(y,t^{2})\Big).
\]
\end{thm}
\begin{proof}
Due to Theorem \ref{t-fnan}, it suffices to show that 
\[
B_{n}(y,t^{2})=\sum_{k=0}^{n}{n \choose k}(1+y)^{k}(1-t^{2})^{n-k}A_{k}(t^{2})
\]
and that 
\[
\frac{(1+y)^{n}}{t}A_{n}(t^{2})=\frac{1}{t}\sum_{k=0}^{n}(-1)^{n-k}{n \choose k}(1-t^{2})^{n-k}B_{k}(y,t^{2}).
\]
Note that Theorem \ref{t-bnan} directly implies the former equation,
whereas a simple application of inclusion-exclusion to Theorem \ref{t-bnan}
implies the latter equation.
\end{proof}
Rather than stating the result of setting $y=1$ in Theorem \ref{t-fnbn},
we give a simpler identity relating the polynomials $F_{n}(t)$ and
$B_{n}(t)$ using (\ref{e-fnan}).
\begin{cor}
\label{c-fnbn} For $n\geq1$, we have
\[
F_{n}(t)=\frac{1}{t}\left(\frac{1+t}{2}\right)^{n}\sum_{k=0}^{n}(-1)^{n-k}{n \choose k}(1-t)^{n-k}B_{k}(t).
\]
\end{cor}
\begin{proof}
By applying inclusion-exclusion to Corollary \ref{c-bnan}, we obtain
\begin{equation}
A_{n}(t)=\frac{1}{2^{n}}\sum_{k=0}^{n}(-1)^{n-k}{n \choose k}(1-t)^{n-k}B_{k}(t).\label{e-anbn}
\end{equation}
Combining (\ref{e-fnan}) and (\ref{e-anbn}) yields the result.
\end{proof}

\section{Noncommutative symmetric functions}

\subsection{Definitions}

In this section, we introduce relevant aspects of the theory of noncommutative
symmetric functions, which were formally introduced by Gelfand, et
al.\ \cite{ncsf1} in 1995 but were implicitly utilized in previous
work, such as \cite{gessel-thesis}.

Throughout this section, fix a field $F$ of characteristic zero.
(We can take $F$ to be $\mathbb{C}$ in subsequent sections.) Then
$F\langle\langle X_{1},X_{2},\dots\rangle\rangle$ is the $F$-algebra
of formal power series in countably many noncommuting variables $X_{1},X_{2},\dots$.
Consider the elements 
\[
\mathbf{h}_{n}\coloneqq\sum_{i_{1}\leq\cdots\leq i_{n}}X_{i_{1}}X_{i_{2}}\cdots X_{i_{n}}
\]
of $F\langle\langle X_{1},X_{2},\dots\rangle\rangle$, which are noncommutative
versions of the complete symmetric functions $h_{n}$. Note that $\mathbf{h}_{n}$
is the noncommutative generating function for weakly increasing words
of length $n$ on the alphabet $\mathbb{P}$. For example, the weakly
increasing word $13449$ is encoded by $X_{1}X_{3}X_{4}^{2}X_{9}$,
which appears as a term in $\mathbf{h}_{5}$. Given a composition
$L=(L_{1},\dots,L_{k})$, we let 
\[
\mathbf{h}_{L}\coloneqq\mathbf{h}_{L_{1}}\cdots\mathbf{h}_{L_{k}}.
\]
Equivalently, 
\[
\mathbf{h}_{L}=\sum_{L}X_{i_{1}}X_{i_{2}}\cdots X_{i_{n}}
\]
where the sum is over all $(i_{1},\dots,i_{n})$ satisfying 
\[
\underset{L_{1}}{\underbrace{i_{1}\leq\cdots\leq i_{L_{1}}}},\underset{L_{2}}{\underbrace{i_{L_{1}+1}\leq\cdots\leq i_{L_{1}+L_{2}}}},\dots,\underset{L_{k}}{\underbrace{i_{L_{1}+\cdots+L_{k-1}+1}\leq\cdots\leq i_{n}}},
\]
so $\mathbf{h}_{L}$ is the noncommutative generating function for
words in $\mathbb{P}$ whose descent composition $K$ satisfies $K\leq L$
in the reverse refinement ordering.

The $F$-algebra generated by the elements $\mathbf{h}_{L}$ is called
the algebra\textbf{ $\mathbf{Sym}$} of \textit{noncommutative symmetric
functions} with coefficients in $F$, which is a subalgebra of $F\langle\langle X_{1},X_{2},\dots\rangle\rangle$.\footnote{The algebra $\mathbf{Sym}$ is defined differently in \cite{ncsf1},
but the algebras are the same.}

Next, let $\mathbf{Sym}_{n}$ be the vector space of noncommutative
symmetric functions homogeneous of degree $n$, so $\mathbf{Sym}_{n}$
is spanned by $\{\mathbf{h}_{L}\}_{L\vDash n}$ where $L\vDash n$
indicates that $L$ is a composition of $n$, and\textbf{ $\mathbf{Sym}$}
is a graded $F$-algebra with 
\[
\mathbf{Sym}=\bigoplus_{n=0}^{\infty}\mathbf{Sym}_{n}.
\]

For a composition $L=(L_{1},\dots,L_{k})$, we also define 
\[
\mathbf{r}_{L}\coloneqq\sum_{L}X_{i_{1}}X_{i_{2}}\cdots X_{i_{n}}
\]
where the sum is over all $(i_{1},\dots,i_{n})$ satisfying 
\[
\underset{L_{1}}{\underbrace{i_{1}\leq\cdots\leq i_{L_{1}}}}>\underset{L_{2}}{\underbrace{i_{L_{1}+1}\leq\cdots\leq i_{L_{1}+L_{2}}}}>\cdots>\underset{L_{k}}{\underbrace{i_{L_{1}+\cdots+L_{k-1}+1}\leq\cdots\leq i_{n}}}.
\]
Then, $\mathbf{r}_{L}$ is the noncommutative generating function
for words on the alphabet $\mathbb{P}$ with descent composition $L$.

Note that 
\begin{equation}
\mathbf{h}_{L}=\sum_{K\leq L}\mathbf{r}_{K},\label{e-hitor}
\end{equation}
and by inclusion-exclusion, 
\begin{equation}
\mathbf{r}_{L}=\sum_{K\leq L}(-1)^{l(L)-l(K)}\mathbf{h}_{K}.\label{e-ritoh}
\end{equation}
Hence, the $\mathbf{r}_{L}$ are noncommutative symmetric functions,
and are in fact noncommutative versions of the ribbon Schur functions
$r_{L}$. 

Since $\mathbf{r}_{L}$ and $\mathbf{r}_{M}$ have no terms in common
for $L\neq M$, it is clear that $\{\mathbf{r}_{L}\}_{L\vDash n}$
is linearly independent. From (\ref{e-hitor}), we see that $\{\mathbf{r}_{L}\}_{L\vDash n}$
spans $\mathbf{Sym}_{n}$, so $\{\mathbf{r}_{L}\}_{L\vDash n}$ is
a basis for $\mathbf{Sym}_{n}$. Because $\{\mathbf{h}_{L}\}_{L\vDash n}$
spans $\mathbf{Sym}_{n}$ and has the same cardinality as $\{\mathbf{r}_{L}\}_{L\vDash n}$,
we conclude that $\{\mathbf{h}_{L}\}_{L\vDash n}$ is also a basis
for $\mathbf{Sym}_{n}$.

Finally, let us consider the noncommutative generating function 
\[
\mathbf{e}_{n}\coloneqq\sum_{i_{1}>\cdots>i_{n}}X_{i_{1}}X_{i_{2}}\cdots X_{i_{n}}
\]
for decreasing words of length $n$ on the alphabet $\mathbb{P}$.
If we let 
\[
\mathbf{h}(x)\coloneqq\sum_{n=0}^{\infty}\mathbf{h}_{n}x^{n}\in\mathbf{Sym}[[x]]
\]
be the generating function for the noncommutative complete symmetric
functions $\mathbf{h}_{n}$ and let 
\[
\mathbf{e}(x)\coloneqq\sum_{n=0}^{\infty}\mathbf{e}_{n}x^{n}
\]
be the generating function for the $\mathbf{e}_{n}$, then it can
be shown (see \cite[p. 38]{gessel-thesis} or \cite[Section 7.3]{ncsf1})
that 
\begin{equation}
\mathbf{e}(x)=\mathbf{h}(-x)^{-1}.\label{e-ehr}
\end{equation}
Since $\mathbf{h}(-x)^{-1}\in\mathbf{Sym}[[x]]$, it follows that
the $\mathbf{e}_{n}$ are also noncommutative symmetric functions.
The generating functions for the ordinary (commutative) elementary
symmetric functions $e_{n}$ and for the complete symmetric functions
$h_{n}$ satisfy the same relation as (\ref{e-ehr}), so we see that
the $\mathbf{e}_{n}$ are noncommutative versions of the elementary
symmetric functions $e_{n}$.

Although we won't need to use this fact in this paper, it is worth
noting that for a composition $L=(L_{1},\dots,L_{k})$ of $n$, we
can define 
\[
\mathbf{e}_{L}\coloneqq\mathbf{e}_{L_{1}}\mathbf{e}_{L_{2}}\cdots\mathbf{e}_{L_{k}}
\]
and $\{\mathbf{e}_{L}\}_{L\models n}$ is a third basis for $\mathbf{Sym}_{n}$.
This can be proven using a noncommutative analogue of the $\omega$
involution for ordinary symmetric functions (see \cite[Section 7.6]{Stanley2001}).

\subsection{Homomorphisms}

Our main results in the next section are obtained by applying certain
homomorphisms to various identities involving noncommutative symmetric
functions. Although noncommutative symmetric functions are generating
functions for words, we shall see that these homomorphisms allow us
to move from the realm of word enumeration to that of permutation
enumeration. We define these homomorphisms $\Phi:\mathbf{Sym}\rightarrow F[[x]]$
by $\Phi(\mathbf{h}_{n})=x^{n}/n!$, and $\Phi_{q}:\mathbf{Sym}\rightarrow F(q)[[x]]$
by $\Phi_{q}(\mathbf{h}_{n})=x^{n}/[n]_{q}!$. Then if $L$ is a composition
of $n$, we have 
\[
\Phi(\mathbf{h}_{L})=\frac{x^{L_{1}}}{L_{1}!}\cdots\frac{x^{L_{k}}}{L_{k}!}={n \choose L}\frac{x^{n}}{n!}
\]
and 
\[
\Phi_{q}(\mathbf{h}_{L})=\frac{x^{L_{1}}}{[L_{1}]_{q}!}\cdots\frac{x^{L_{k}}}{[L_{k}]_{q}!}={n \choose L}_{\!\!q}\frac{x^{n}}{[n]_{q}!}.
\]

For our proofs, we also need to determine the effect of $\Phi$ and
$\Phi_{q}$ on $\mathbf{r}_{L}$, $\mathbf{h}(1)=\sum_{n=0}^{\infty}\mathbf{h}_{n}$,
and $\mathbf{e}(1)=\sum_{n=0}^{\infty}\mathbf{e}_{n}$. Recall that
$\beta(L)$ is the number of $n$-permutations with descent composition
$L$ and that $\beta_{q}(L)$ is the polynomial counting $n$-permutations
with descent composition $L$ by inversion number.
\begin{lem}
\label{l-Phi} \leavevmode

\begin{enumerate}
\item [\normalfont{(a)}] Let $L$ be a composition of $n$. Then $\Phi(\mathbf{r}_{L})=\beta(L)x^{n}/n!$.
\item [\normalfont{(b)}] $\Phi(\mathbf{h}(1))=e^{x}$.
\item [\normalfont{(c)}] $\Phi(\mathbf{e}(1))=e^{x}$.
\end{enumerate}
\end{lem}
\begin{proof}
Part (a):{\allowdisplaybreaks}
\begin{align*}
\Phi(\mathbf{r}_{L}) & =\Phi\Big(\sum_{K\leq L}(-1)^{l(L)-l(K)}\mathbf{h}_{K}\Big),\:\mbox{by (\ref{e-ritoh})}\\
 & =\sum_{K\leq L}(-1)^{l(L)-l(K)}\Phi(\mathbf{h}_{K})\\
 & =\sum_{K\leq L}(-1)^{l(L)-l(K)}{n \choose K}\frac{x^{n}}{n!}\\
 & =\beta(L)\frac{x^{n}}{n!},\:\mbox{by (\ref{e-betaL})}.
\end{align*}
Part (b):
\begin{align*}
\Phi(\mathbf{h}(1)) & =\Phi\Big(\sum_{n=0}^{\infty}\mathbf{h}_{n}\Big)=\sum_{n=0}^{\infty}\Phi(\mathbf{h}_{n})=\sum_{n=0}^{\infty}\frac{x^{n}}{n!}=e^{x}.
\end{align*}
Part (c):
\begin{align*}
\Phi(\mathbf{e}(1)) & =\Phi(\mathbf{h}(-1)^{-1}),\:\mbox{by (\ref{e-ehr})}\\
 & =\Big(\sum_{n=0}^{\infty}(-1)^{n}\Phi(\mathbf{h}_{n})\Big)^{-1}\\
 & =\Big(\sum_{n=0}^{\infty}\frac{(-x)^{n}}{n!}\Big)^{-1}\\
 & =(e^{-x})^{-1}\\
 & =e^{x}.\qedhere
\end{align*}
\end{proof}
Consider the $q$-exponential function 
\[
\exp_{q}(x)\coloneqq\sum_{n=0}^{\infty}\frac{x^{n}}{[n]_{q}!}
\]
and its variant 
\[
\Exp_{q}(x)\coloneqq\sum_{n=0}^{\infty}q^{{n \choose 2}}\frac{x^{n}}{[n]_{q}!},
\]
both $q$-analogues of the classical exponential function $e^{x}$.
It is well known that $\Exp_{q}(x)=(\exp_{q}(-x))^{-1}$. Using the
same ideas as in the proof of Lemma \ref{l-Phi}, it is easy to verify
the following lemma.
\begin{lem}
\label{l-Phiq} \leavevmode

\begin{enumerate}
\item [\normalfont{(a)}] Let $L$ be a composition of $n$. Then $\Phi_{q}(\mathbf{r}_{L})=\beta_{q}(L)x^{n}/[n]_{q}!$.
\item [\normalfont{(b)}] $\Phi_{q}(\mathbf{h}(1))=\exp_{q}(x)$.
\item [\normalfont{(c)}] $\Phi_{q}(\mathbf{e}(1))=\Exp_{q}(x)$.
\end{enumerate}
\end{lem}

\section{Main results}

\subsection{Peaks and descents}

Consider the polynomial 
\[
P_{n}^{(\pk,\des)}(y,t)\coloneqq\sum_{\pi\in\mathfrak{S}_{n}}y^{\pk(\pi)+1}t^{\des(\pi)+1}
\]
which refines the Eulerian polynomial $A_{n}(t)$ and the peak polynomial
$P_{n}^{\pk}(t)$. We prove in our first theorem an identity expressing
$P_{n}^{(\pk,\des)}(y,t)$ in terms of $A_{n}(t)$. To do so, we need
to establish the following noncommutative symmetric function identity.
\begin{lem}
\textup{\label{l-pkdesncsf}
\begin{multline*}
(1-t\mathbf{e}(yx)\mathbf{h}(x))^{-1}=\\
\frac{1}{1-t}+\sum_{n=1}^{\infty}\sum_{L\vDash n}\frac{t^{\pk(L)+1}(y+t)^{\des(L)-\pk(L)}(1+yt)^{n-\pk(L)-\des(L)-1}(1+y)^{2\pk(L)+1}}{(1-t)^{n+1}}x^{n}\mathbf{r}_{L}
\end{multline*}
}
\end{lem}
\begin{proof}
Let $\underline{\mathbb{P}}=\{\underline{1},\underline{2},\underline{3},\dots\}$
denote the set of positive integers decorated with underlines, endowed
with the usual total ordering of $\mathbb{P}$. Let us say that a
word $w$ on the alphabet $\mathbb{P}\cup\underline{\mathbb{P}}\cup\{|\}$
(that is, the positive integers, underlined positive integers, and
a vertical bar) is a \textit{peak word} if $w$ can be written as
a sequence of subwords of the form $w_{1}w_{2}|$ where $w_{1}$ is
a (possibly empty) strictly decreasing word containing only letters
from $\underline{\mathbb{P}}$ and $w_{2}$ is a (possibly empty)
weakly increasing word containing only letters from $\mathbb{P}$.
For example, 
\begin{equation}
\underline{8642}11|457|\underline{931}||\underline{1}2338|56|||\underline{942}788|\label{e-pkword}
\end{equation}
is a peak word. It is clear that the left-hand side of the given equation
counts peak words where $t$ is weighting the number of bars, $y$
is weighting the number of underlined letters, and $x$ is weighting
the length of the underlying $\mathbb{P}$-word. We want to show that
the right-hand side also counts peak words with the same weights.

Let us say that a peak word is \textit{minimal} if it is impossible
to remove bars from it to yield a peak word. Given a word in $\mathbb{P}$,
there is a unique minimal peak word corresponding to every possible
choice of underlines. Indeed, if $w$ is a word in $\mathbb{P}$ with
a given choice of underlines (that is, if $w$ is a word in $\mathbb{P}\cup\underline{\mathbb{P}}$),
then a minimal peak word corresponding to $w$ must have no bar at
the beginning and a bar at the end, and whether or not there needs
to be a bar between two letters $a$ and $b$ is completely determined
by whether $a>b$, whether $a$ is underlined, and whether $b$ is
underlined. Moreover, adding bars to a peak word yields another peak
word, so every peak word can be obtained from a unique minimal peak
word by adding bars. For example, the minimal peak word corresponding
to 
\[
\underline{8642}11457\underline{9311}233856\underline{942}788
\]
is 
\[
\underline{8642}11457|\underline{931}|\underline{1}2338|56|\underline{942}788|,
\]
which is the unique minimal peak word from which we can obtain (\ref{e-pkword})
as they share the same underlying $\mathbb{P}$-word and choice of
underlines.

We show that 
\begin{equation}
t(t+yt)^{\pk(L)}(1+y)^{\pk(L)+1}(y+t)^{\des(L)-\pk(L)}(1+yt)^{n-\des(L)-\pk(L)-1}x^{n}\mathbf{r}_{L}\label{e-pkncsfn}
\end{equation}
counts nonempty minimal peak words with descent composition $L\vDash n$.
Every term in $\mathbf{r}_{L}$ corresponds to a $\mathbb{P}$-word
with descent composition $L$, and we give it a choice of underlines
and insert necessary bars. As our working example, take the $\mathbb{P}$-word
$11375438876544579756673$.

\begin{enumerate}
\item There must be a bar at the end, hence the initial factor $t$:
\[
11375438876544579756673|.
\]
\item For each letter corresponding to a peak, we choose whether or not
to underline it. If we do underline it, then we insert a bar immediately
before it; otherwise, we insert a bar immediately after it. This corresponds
to the $(t+yt)^{\pk(L)}$ factor. For example, we may have
\[
113|\underline{7}54388|7654457|\underline{9}7566|\underline{7}3|.
\]
\item The above step divides our word into $\pk(L)+1$ segments, separated
by bars. Take the left-most smallest letter of each segment and choose
whether or not to underline it; this gives the $(1+y)^{\pk(L)+1}$
factor. For example, we may have 
\[
\underline{1}13|\underline{7}54\underline{3}88|7654457|\underline{9}7\underline{5}66|\underline{7}3|.
\]
Note that this step determines whether the left-most smallest letter
in each segment is to be part of the underlined decreasing subword
or the non-underlined weakly increasing subword.
\item Take each letter corresponding to a descent that is not a peak and
choose to either underline it or to add a bar after it; this gives
$(y+t)^{\des(L)-\pk(L)}$. For example, we may have 
\[
\underline{1}13|\underline{75}4|\underline{3}88|\underline{765}4457|\underline{9}7|\underline{5}66|\underline{7}3|.
\]
This step eliminates instances of underlined letters separated by
non-underlined letters in the same segment, and it is evident that
this gives the minimal peak word corresponding to our current choice
of underlines.
\item Finally, iterate through every letter that is (a) not the final letter
of the word, (b) not corresponding to a descent, and (c) not followed
immediately by a letter corresponding to a peak, and choose either
to do nothing or to underline the next letter and add a bar in between
the two letters; this gives $(1+yt)^{n-\des(L)-\pk(L)-1}$. For example,
we may have 
\[
\underline{1}1|\underline{3}|\underline{75}4|\underline{3}|\underline{8}8|\underline{765}4457|\underline{9}7|\underline{5}6|\underline{6}|\underline{7}3|.
\]
Note that adding these underlines requires the corresponding bars
to be placed, so the result is still a minimal peak word.
\end{enumerate}
Through these steps, we have considered whether to underline each
letter in the word, so in fact (\ref{e-pkncsfn}) accounts for the
unique minimal peak word corresponding to each choice of underlines,
and thus counts all minimal peak words with descent composition $L\vDash n$.

Observe that (\ref{e-pkncsfn}) is equal to 
\[
t^{\pk(L)+1}(y+t)^{\des(L)-\pk(L)}(1+yt)^{n-\pk(L)-\des(L)-1}(1+y)^{2\pk(L)+1}x^{n}\mathbf{r}_{L},
\]
which appears in the statement of this lemma. Dividing by $(1-t)^{n+1}$
corresponds to inserting any number of bars in the $n+1$ possible
positions, which allows us to move from nonempty minimal peak words
to all peak words except those that only consist of bars, which are
accounted for by the $1/(1-t)$ term at the beginning. Hence the lemma
is proven.
\end{proof}
We remark that Lemma \ref{l-pkdesncsf} (as well as the noncommutative
symmetric function identities stated in the next several subsections)
can also be proven using the present author's ``generalized run theorem''
\cite{Zhuang2016} and making appropriate substitutions, but we prefer
the combinatorial proof given above.

The following is our main result relating the Eulerian polynomials
and the $(\pk,\des)$ polynomials.
\begin{thm}
\label{t-pkdes} For $n\geq1$, we have
\begin{equation}
A_{n}(t)=\left(\frac{1+yt}{1+y}\right)^{n+1}P_{n}^{(\pk,\des)}\left(\frac{(1+y)^{2}t}{(y+t)(1+yt)},\frac{y+t}{1+yt}\right).\label{e-apkdes}
\end{equation}
Equivalently, 
\begin{equation}
P_{n}^{(\pk,\des)}(y,t)=\left(\frac{1+u}{1+uv}\right)^{n+1}A_{n}(v)\label{e-pkdesa}
\end{equation}
where $u=\frac{1+t^{2}-2yt-(1-t)\sqrt{(1+t)^{2}-4yt}}{2(1-y)t}$ and
$v=\frac{(1+t)^{2}-2yt-(1+t)\sqrt{(1+t)^{2}-4yt}}{2yt}.$
\end{thm}
Note that evaluating (\ref{e-apkdes}) at $y=1$ recovers Stembridge's
identity 
\[
A_{n}(t)=\left(\frac{1+t}{2}\right)^{n+1}P_{n}^{\pk}\left(\frac{4t}{(1+t)^{2}}\right)
\]
mentioned in the introduction of this paper.
\begin{proof}
Taking Lemma \ref{l-pkdesncsf}, evaluating at $x=1$, and applying
the homomorphism $\Phi$ yields 
\begin{multline*}
\frac{1}{1-te^{(1+y)x}}=\\
\frac{1}{1-t}+\sum_{n=1}^{\infty}\sum_{\pi\in\mathfrak{S}_{n}}\frac{t^{\pk(\pi)+1}(y+t)^{\des(\pi)-\pk(\pi)}(1+yt)^{n-\pk(\pi)-\des(\pi)-1}(1+y)^{2\pk(\pi)+1}}{(1-t)^{n+1}}\frac{x^{n}}{n!}
\end{multline*}
by Lemma \ref{l-Phi}. Rearranging some terms yields 
\begin{multline*}
\frac{1}{1-te^{(1+y)x}}=\\
\frac{1}{1-t}+\sum_{n=1}^{\infty}\sum_{\pi\in\mathfrak{S}_{n}}\frac{1}{1+y}\left(\frac{1+yt}{1-t}\right)^{n+1}\left(\frac{(1+y)^{2}t}{(y+t)(1+yt)}\right)^{\pk(\pi)+1}\left(\frac{y+t}{1+yt}\right)^{\des(\pi)+1}\frac{x^{n}}{n!}.
\end{multline*}
Multiplying both sides by $1-t$ and then replacing $x$ by $(1-t)x/(1+y)$
yields 
\[
\frac{1-t}{1-te^{(1-t)x}}=1+\sum_{n=1}^{\infty}\sum_{\pi\in\mathfrak{S}_{n}}\left(\frac{1+yt}{1+y}\right)^{n+1}\left(\frac{(1+y)^{2}t}{(y+t)(1+yt)}\right)^{\pk(\pi)+1}\left(\frac{y+t}{1+yt}\right)^{\des(\pi)+1}\frac{x^{n}}{n!}.
\]
Note that the left-hand side is the exponential generating function
for the Eulerian polynomials; thus equating the coefficients of $x^{n}/n!$
gives (\ref{e-apkdes}). 

Finally, (\ref{e-pkdesa}) can be obtained by setting $u=\frac{(1+y)^{2}t}{(y+t)(1+yt)}$
and $v=\frac{y+t}{1+yt}$, solving for $y$ and $t$ (which can be
done using a computer algebra system such as Maple), and simplifying.\footnote{We exchanged $u$ and $v$ with $y$ and $t$, respectively, in the
statement of (\ref{e-pkdesa}) in this theorem, so that the $(\pk,\des)$
polynomial would have variables $y$ and $t$ as in its definition.
The same is done for all subsequent results involving similar substitutions.}
\end{proof}
We give a combinatorial proof of this result using the modified Foata--Strehl
action in Subsection 5.1.

Next, we obtain a similar result for the $q$-analogue of the $(\pk,\des)$
polynomial 
\[
P_{n}^{(\inv,\pk,\des)}(q,y,t)\coloneqq\sum_{\pi\in\mathfrak{S}_{n}}q^{\inv(\pi)}y^{\pk(\pi)+1}t^{\des(\pi)+1}
\]
also keeping track of the inversion number.
\begin{thm}
\label{t-pkdesq} We have \textup{
\begin{multline}
\frac{1-t}{1-t\Exp_{q}(yx)\exp_{q}(x)}=\\
1+\sum_{n=1}^{\infty}\frac{(1+yt)^{n+1}}{(1+y)(1-t)^{n}}P_{n}^{(\inv,\pk,\des)}\left(q,\frac{(1+y)^{2}t}{(y+t)(1+yt)},\frac{y+t}{1+yt}\right)\frac{x^{n}}{[n]_{q}!}.\label{e-qpkdes1}
\end{multline}
}Equivalently, 
\[
\sum_{n=1}^{\infty}P_{n}^{(\inv,\pk,\des)}(q,y,t)\frac{x^{n}}{[n]_{q}!}=\frac{v(1+u)}{1+uv}\frac{\Exp_{q}\left(\frac{u(1-v)}{1+uv}x\right)\exp_{q}\left(\frac{1-v}{1+uv}x\right)-1}{1-v\Exp_{q}\left(\frac{u(1-v)}{1+uv}x\right)\exp_{q}\left(\frac{1-v}{1+uv}x\right)}
\]
where $u=\frac{1+t^{2}-2yt-(1-t)\sqrt{(1+t)^{2}-4yt}}{2(1-y)t}$ and
$v=\frac{(1+t)^{2}-2yt-(1+t)\sqrt{(1+t)^{2}-4yt}}{2yt}.$ 
\end{thm}
\begin{proof}
We follow the proof of Theorem \ref{t-pkdes}, but apply the homomorphism
$\Phi_{q}$ instead of $\Phi$, which by Lemma \ref{l-Phiq} yields
\begin{multline*}
\frac{1}{1-t\Exp_{q}(yx)\exp_{q}(x)}=\frac{1}{1-t}\\
+\sum_{n=1}^{\infty}\sum_{\pi\in\mathfrak{S}_{n}}q^{\inv(\pi)}\frac{t^{\pk(\pi)+1}(y+t)^{\des(\pi)-\pk(\pi)}(1+yt)^{n-\pk(\pi)-\des(\pi)-1}(1+y)^{2\pk(\pi)+1}}{(1-t)^{n+1}}\frac{x^{n}}{[n]_{q}!}.
\end{multline*}
Multiplying both sides by $1-t$ and rearranging some terms yields
(\ref{e-qpkdes1}).

Next, we replace $x$ by $(1-t)x/(1+yt)$ to get 
\begin{multline*}
\frac{1-t}{1-t\Exp_{q}\left(\frac{y(1-t)}{1+yt}x\right)\exp_{q}\left(\frac{1-t}{1+yt}x\right)}=\\
1+\frac{1+yt}{1+y}\sum_{n=1}^{\infty}P_{n}^{(\inv,\pk,\des)}\left(q,\frac{(1+y)^{2}t}{(y+t)(1+yt)},\frac{y+t}{1+yt}\right)\frac{x^{n}}{[n]_{q}!}.
\end{multline*}
Making the same substitutions yields 
\[
1+\frac{1+uv}{1+u}\sum_{n=1}^{\infty}P_{n}^{(\inv,\pk,\des)}(q,y,t)\frac{x^{n}}{[n]_{q}!}=\frac{1-v}{1-v\Exp_{q}\left(\frac{u(1-v)}{1+uv}x\right)\exp_{q}\left(\frac{1-v}{1+uv}x\right)}
\]
where $u=\frac{1+t^{2}-2yt-(1-t)\sqrt{(1+t)^{2}-4yt}}{2(1-y)t}$ and
$v=\frac{(1+t)^{2}-2yt-(1+t)\sqrt{(1+t)^{2}-4yt}}{2yt}.$ Subtracting
both sides by 1 and dividing by $(1+uv)/(1+u)$ completes the proof.
\end{proof}
Unfortunately, we cannot express $P_{n}^{(\inv,\pk,\des)}(q,y,t)$
in terms of the $q$-Eulerian polynomial 
\[
A_{n}(q,t)\coloneqq\sum_{\pi\in\mathfrak{S}_{n}}q^{\inv(\pi)}t^{\des(\pi)+1},
\]
but we can recover the known $q$-exponential generating function
for the $q$-Eulerian polynomials from the above result.
\begin{cor}
\textup{\label{c-qegf}
\begin{align*}
\sum_{n=0}^{\infty}A_{n}(q,t)\frac{x^{n}}{[n]_{q}!} & =\frac{1-t}{1-t\exp_{q}((1-t)x)}
\end{align*}
}
\end{cor}
\begin{proof}
Take (\ref{e-qpkdes1}), set $y=0$, and replace $x$ with $(1-t)x$.
\end{proof}
Theorem \ref{t-pkdesq} also specializes to a corresponding result
for the $(\inv,\pk)$ polynomial 
\[
P_{n}^{(\inv,\pk)}(q,t)\coloneqq\sum_{\pi\in\mathfrak{S}_{n}}q^{\inv(\pi)}t^{\pk(\pi)+1}.
\]

\begin{cor}
\label{c-pkq} We have 
\begin{equation}
\frac{1-t}{1-t\Exp_{q}(x)\exp_{q}(x)}=1+\sum_{n=1}^{\infty}\frac{(1+t)^{n+1}}{2(1-t)^{n}}P_{n}^{(\inv,\pk)}\left(q,\frac{4t}{(1+t)^{2}}\right)\frac{x^{n}}{[n]_{q}!}.\label{e-qpeak}
\end{equation}
Equivalently, 
\begin{equation}
\sum_{n=1}^{\infty}P_{n}^{(\inv,\pk)}(q,t)\frac{x^{n}}{[n]_{q}!}=\frac{2v}{1+v}\frac{\Exp_{q}\left(\frac{1-v}{1+v}x\right)\exp_{q}\left(\frac{1-v}{1+v}x\right)-1}{1-v\Exp_{q}\left(\frac{1-v}{1+v}x\right)\exp_{q}\left(\frac{1-v}{1+v}x\right)}\label{e-qpeak2}
\end{equation}
where $v=\frac{2}{t}(1-\sqrt{1-t})-1$.
\end{cor}
\begin{proof}
Equation (\ref{e-qpeak}) is obtained by taking (\ref{e-qpkdes1})
and setting $y=1$. Then (\ref{e-qpeak2}) follows by replacing $x$
with $x(1-t)/(1+t)$, making an appropriate substitution, and rearranging
some terms.
\end{proof}

\subsection{Left peaks and descents}

In this subsection, we study the $(\lpk,\des)$ polynomials 
\[
P_{n}^{(\lpk,\des)}(y,t)\coloneqq\sum_{\pi\in\mathfrak{S}_{n}}y^{\lpk(\pi)}t^{\des(\pi)}
\]
and their $q$-analogues 
\[
P_{n}^{(\inv,\lpk,\des)}(q,y,t)\coloneqq\sum_{\pi\in\mathfrak{S}_{n}}q^{\inv(\pi)}y^{\lpk(\pi)}t^{\des(\pi)}.
\]
Using the same method as before, we obtain analogues of Theorems \ref{t-pkdes}
and \ref{t-pkdesq} for left peaks and descents, as well as a connection
to the refined type B Eulerian polynomials introduced in Subsection
2.3.

We begin by proving a suitable noncommutative symmetric function identity.
\begin{lem}
\textup{\label{l-lpkdesncsf}
\begin{multline*}
\mathbf{h}(x)(1-t\mathbf{e}(yx)\mathbf{h}(x))^{-1}=\\
\frac{1}{1-t}+\sum_{n=1}^{\infty}\sum_{L\vDash n}\frac{t^{\lpk(L)}(y+t)^{\des(L)-\lpk(L)}(1+yt)^{n-\lpk(L)-\des(L)}(1+y)^{2\lpk(L)}}{(1-t)^{n+1}}x^{n}\mathbf{r}_{L}
\end{multline*}
}
\end{lem}
\begin{proof}
Let us say that a word $w$ on the alphabet $\mathbb{P}\cup\underline{\mathbb{P}}\cup\{|\}$
is a \textit{left peak word} if $w$ begins with a (possibly empty)
weakly increasing subword containing only letters from $\mathbb{P}$,
followed by a sequence of subwords of the form $|w_{1}w_{2}$ where
$w_{1}$ is a (possibly empty) strictly decreasing word containing
only letters from $\underline{\mathbb{P}}$ and $w_{2}$ is a (possibly
empty) weakly increasing word containing only letters from $\mathbb{P}$.
The left-hand side of the given equation counts left peak words where
$t$ is weighting the number of bars, $y$ is weighting the number
of underlined letters, and $x$ is weighting the length of the underlying
$\mathbb{P}$-word. We want show that the right-hand side also counts
left peak words with the same weights.

Call a left peak word $w$ \textit{minimal} if it is impossible to
remove bars from $w$ to yield a left peak word. Similar to peak words
in the proof of Lemma \ref{l-pkdesncsf}, every left peak word can
be obtained from only one minimal left peak word, which is the only
minimal left peak word on those letters with the same choice of underlines.
We claim that 
\begin{equation}
(t+yt)^{\lpk(L)}(1+y)^{\lpk(L)}(y+t)^{\des(L)-\lpk(L)}(1+yt)^{n-\des(L)-\lpk(L)}x^{n}\mathbf{r}_{L}\label{e-lpkncsfn}
\end{equation}
counts nonempty minimal peak words with descent composition $L\vDash n$.
Every term in $\mathbf{r}_{L}$ corresponds to a $\mathbb{P}$-word
with descent composition $L$, and we give it a choice of underlines
and insert bars in a similar way as in the proof of Lemma \ref{l-pkdesncsf}:

\begin{enumerate}
\item For each letter corresponding to a left peak, we choose whether or
not to underline it. If we do underline it, then we insert a bar immediately
before it; otherwise, we insert a bar immediately after it. This corresponds
to the $(t+yt)^{\lpk(L)}$ factor. 
\item If the first letter corresponds to a left peak and was underlined,
then the bars inserted in the above step divide our word into $\lpk(L)$
segments. In this case, take the left-most smallest letter of each
segment and choose whether or not to underline it. Otherwise, the
bars divide our word into $\lpk(L)+1$ segments, in which case we
take the left-most smallest letter of each but the first segment and
choose whether or not to underline it. This gives the $(1+y)^{\lpk(L)}$
factor.
\item Take each letter corresponding to a descent that is not a left peak
and choose to either underline it or to add a bar after it; this gives
$(y+t)^{\des(L)-\lpk(L)}$. As in the proof of Lemma \ref{l-pkdesncsf},
this step eliminates underlined letters separated by non-underlined
letters appearing in the same segment, and gives a minimal left peak
word corresponding to our current choice of underlines.
\item Finally, iterate through every letter that is (a) not the final letter
of the word, (b) not corresponding a descent, and (c) not followed
by a letter corresponding to a left peak, and choose either to do
nothing or to underline the next letter and add a bar in between the
two letters. In addition, if the first letter does not correspond
to a left peak, then choose to either do nothing or to underline the
first letter and prepend a bar. This gives $(1+yt)^{n-\des(L)-\lpk(L)-1}$,
and the result is still a minimal left peak word as the new bars are
necessary to accomodate the new underlines.
\end{enumerate}
Through these steps, we have considered whether to underline each
letter in the word, so (\ref{e-lpkncsfn}) counts every minimal left
peak word with descent composition $L\vDash n$. Dividing by $(1-t)^{n+1}$
allows us to insert any number of bars in any of the $n+1$ possible
positions, thus creating left peak words from minimal left peak words,
and the $1/(1-t)$ term accounts for words containing only bars.
\end{proof}
\begin{thm}
\label{t-lpkdes} For $n\geq0$, we have
\begin{equation}
\sum_{k=0}^{n}{n \choose k}(1+y)^{k}(1-t)^{n-k}A_{k}(t)=(1+yt)^{n}P_{n}^{(\lpk,\des)}\left(\frac{(1+y)^{2}t}{(y+t)(1+yt)},\frac{y+t}{1+yt}\right).\label{e-alpkdes}
\end{equation}
Equivalently, 
\begin{equation}
P_{n}^{(\lpk,\des)}(y,t)=\frac{1}{(1+uv)^{n}}\sum_{k=0}^{n}{n \choose k}(1+u)^{k}(1-v)^{n-k}A_{k}(v)\label{e-lpkdesa}
\end{equation}
where $u=\frac{1+t^{2}-2yt-(1-t)\sqrt{(1+t)^{2}-4yt}}{2(1-y)t}$ and
$v=\frac{(1+t)^{2}-2yt-(1+t)\sqrt{(1+t)^{2}-4yt}}{2yt}.$ 
\end{thm}
As with Theorem \ref{t-pkdes}, evaluating (\ref{e-alpkdes}) at $y=1$
recovers a known result, Petersen's identity 
\[
\sum_{k=0}^{n}{n \choose k}2^{k}(1-t)^{n-k}A_{k}(t)=(1+t)^{n}P_{n}^{\lpk}\left(\frac{4t}{(1+t)^{2}}\right)
\]
in this case.
\begin{proof}
Taking Lemma \ref{l-lpkdesncsf}, evaluating at $x=1$, applying the
homomorphism $\Phi$, and rearranging some terms yields 
\begin{align*}
\frac{e^{x}}{1-te^{(1+y)x}} & =\frac{1}{1-t}+\sum_{n=1}^{\infty}\sum_{\pi\in\mathfrak{S}_{n}}\frac{(1+yt)^{n}}{(1-t)^{n+1}}\left(\frac{(1+y)^{2}t}{(y+t)(1+yt)}\right)^{\lpk(\pi)}\left(\frac{y+t}{1+yt}\right)^{\des(\pi)}\frac{x^{n}}{n!}.
\end{align*}
Multiplying both sides by $1-t$ and then replacing $x$ by $(1-t)x/(1+y)$
yields 
\begin{align*}
\frac{1-t}{1-te^{(1-t)x}}e^{\frac{1-t}{1+y}x} & =\sum_{n=0}^{\infty}\sum_{\pi\in\mathfrak{S}_{n}}\left(\frac{1+yt}{1+y}\right)^{n}\left(\frac{(1+y)^{2}t}{(y+t)(1+yt)}\right)^{\lpk(\pi)}\left(\frac{y+t}{1+yt}\right)^{\des(\pi)}\frac{x^{n}}{n!}
\end{align*}
Moreover, 
\begin{align*}
\frac{1-t}{1-te^{(1-t)x}}e^{\frac{1-t}{1+y}x} & =\Big(\sum_{n=0}^{\infty}A_{n}(t)\frac{x^{n}}{n!}\Big)\Big(\sum_{n=0}^{\infty}\left(\frac{1-t}{1+y}\right)^{n}\frac{x^{n}}{n!}\Big)\\
 & =\sum_{n=0}^{\infty}\sum_{k=0}^{n}{n \choose k}A_{k}(t)\left(\frac{1-t}{1+y}\right)^{n-k}\frac{x^{n}}{n!},
\end{align*}
so 
\begin{multline*}
\sum_{n=0}^{\infty}\sum_{k=0}^{n}{n \choose k}A_{k}(t)\left(\frac{1-t}{1+y}\right)^{n-k}\frac{x^{n}}{n!}=\\
\sum_{n=0}^{\infty}\sum_{\pi\in\mathfrak{S}_{n}}\left(\frac{1+yt}{1+y}\right)^{n}\left(\frac{(1+y)^{2}t}{(y+t)(1+yt)}\right)^{\lpk(\pi)}\left(\frac{y+t}{1+yt}\right)^{\des(\pi)}\frac{x^{n}}{n!}.
\end{multline*}
Equating the coefficients of $x^{n}/n!$ and rearranging some terms
gives (\ref{e-alpkdes}). Then (\ref{e-lpkdesa}) can be obtained
by making the same substitutions as in the proof of Theorem \ref{t-pkdes}.
\end{proof}
Now, for the $q$-analogue.
\begin{thm}
\label{t-lpkdesq} We have \textup{
\begin{align}
\frac{(1-t)\exp_{q}(x)}{1-t\Exp_{q}(yx)\exp_{q}(x)} & =\sum_{n=0}^{\infty}\left(\frac{1+yt}{1-t}\right)^{n}P_{n}^{(\inv,\lpk,\des)}\left(q,\frac{(1+y)^{2}t}{(y+t)(1+yt)},\frac{y+t}{1+yt}\right)\frac{x^{n}}{[n]_{q}!}.\label{e-qlpkdes1}
\end{align}
}Equivalently, \textup{
\begin{equation}
\sum_{n=0}^{\infty}P_{n}^{(\inv,\lpk,\des)}(q,y,t)\frac{x^{n}}{[n]_{q}!}=\frac{(1-v)\exp_{q}\left(\frac{1-v}{1+uv}x\right)}{1-v\Exp_{q}\left(\frac{u(1-v)}{1+uv}x\right)\exp_{q}\left(\frac{1-v}{1+uv}x\right)}\label{e-qlpkdes2}
\end{equation}
}where $u=\frac{1+t^{2}-2yt-(1-t)\sqrt{(1+t)^{2}-4yt}}{2(1-y)t}$
and $v=\frac{(1+t)^{2}-2yt-(1+t)\sqrt{(1+t)^{2}-4yt}}{2yt}.$ 
\end{thm}
\begin{proof}
Apply the homomorphism $\Phi_{q}$ to Lemma \ref{l-lpkdesncsf} evaluated
at $x=1$; then multiplying both sides by $1-t$ yields (\ref{e-qlpkdes1}).

Next, replace $x$ by $(1-t)x/(1+yt)$ to get 
\begin{align*}
\frac{(1-t)\exp_{q}\left(\frac{1-t}{1+yt}x\right)}{1-t\Exp_{q}\left(\frac{y(1-t)}{1+yt}x\right)\exp_{q}\left(\frac{1-t}{1+yt}x\right)} & =\sum_{n=0}^{\infty}P_{n}^{(\inv,\lpk,\des)}\left(q,\frac{(1+y)^{2}t}{(y+t)(1+yt)},\frac{y+t}{1+yt}\right)\frac{x^{n}}{[n]_{q}!}.
\end{align*}
Making the same substitutions as before yields (\ref{e-qlpkdes2}).
\end{proof}
We note that taking (\ref{e-qlpkdes1}) and evaluating at $y=0$ gives
\[
\frac{(1-t)\exp_{q}((1-t)x)}{1-t\exp_{q}((1-t)x)}=\sum_{n=0}^{\infty}q^{\inv(\pi)}t^{\des(\pi)}\frac{x^{n}}{[n]_{q}!},
\]
which is equivalent to Corollary \ref{c-qegf}. Evaluating at $y=1$,
on the other hand, gives us a result for the $(\inv,\lpk)$ polynomials
\[
P_{n}^{(\inv,\lpk)}(q,t)\coloneqq\sum_{\pi\in\mathfrak{S}_{n}}q^{\inv(\pi)}t^{\lpk(\pi)}.
\]
 
\begin{cor}
\label{c-lpkq} We have \textup{
\begin{align*}
\frac{(1-t)\exp_{q}(x)}{1-t\Exp_{q}(x)\exp_{q}(x)} & =\sum_{n=0}^{\infty}\left(\frac{1+t}{1-t}\right)^{n}P_{n}^{(\inv,\lpk)}\left(q,\frac{4t}{(1+t)^{2}}\right)\frac{x^{n}}{[n]_{q}!}.
\end{align*}
}Equivalently, 
\[
\sum_{n=0}^{\infty}P_{n}^{(\inv,\lpk)}(q,t)\frac{x^{n}}{[n]_{q}!}=\frac{(1-v)\exp_{q}\left(\frac{1-v}{1+v}x\right)}{1-v\Exp_{q}\left(\frac{1-v}{1+v}x\right)\exp_{q}\left(\frac{1-v}{1+v}x\right)}
\]
where $v=\frac{2}{t}(1-\sqrt{1-t})-1$.
\end{cor}
Lastly, we state an identity connecting the $(\lpk,\des)$ polynomials
with the refined type B Eulerian polynomials $B_{n}(y,t)=\sum_{\pi\in\mathfrak{B}_{n}}y^{\neg(\pi)}t^{\des_{B}(\pi)}.$
\begin{thm}
\label{t-lpkdesb} For all $n\geq0$, we have 
\begin{equation}
B_{n}(y,t)=(1+yt)^{n}P_{n}^{(\lpk,\des)}\left(\frac{(1+y)^{2}t}{(y+t)(1+yt)},\frac{y+t}{1+yt}\right).\label{e-lpkdesb1}
\end{equation}
Equivalently, 
\begin{equation}
P_{n}^{(\lpk,\des)}\left(y,t\right)=\frac{B_{n}(u,v)}{(1+uv)^{n}}\label{e-lpkdesb2}
\end{equation}
where $u=\frac{1+t^{2}-2yt-(1-t)\sqrt{(1+t)^{2}-4yt}}{2(1-y)t}$ and
$v=\frac{(1+t)^{2}-2yt-(1+t)\sqrt{(1+t)^{2}-4yt}}{2yt}.$ 
\end{thm}
By setting $y=1$ in (\ref{e-lpkdesb1}), we recover the identity
\[
B_{n}(t)=(1+t)^{n}P_{n}^{\lpk}\left(\frac{4t}{(1+t)^{2}}\right)
\]
where $B_{n}(t)=\sum_{\pi\in\mathfrak{B}_{n}}t^{\des_{B}(\pi)}$,
another result of Petersen \cite{Petersen2007}.
\begin{proof}
Observe that (\ref{e-lpkdesb1}) follows immediately from Theorems
\ref{t-bnan} and \ref{t-lpkdes}. Making the same substitutions as
before yields (\ref{e-lpkdesb2}).
\end{proof}
We give a combinatorial proof of this result in Subsection 5.3.

\subsection{Up-down runs and descents}

Our remaining aim in this section is to prove analogous results for
the number of up-down runs $\udr$ and the joint distribution of $\udr$
and $\des$. In particular, we express the polynomial 
\[
P_{n}^{\udr}(y,t)\coloneqq\sum_{\pi\in\mathfrak{S}_{n}}t^{\udr(\pi)}
\]
in terms of the $n$th Eulerian polynomial, give a $q$-exponential
generating function for its $q$-analogue, and relate it to the distribution
of the flag descent number over $\mathfrak{B}_{n}$. Due to technical
constraints that will become apparent later, we cannot do the same
with the polynomial 
\[
P_{n}^{(\udr,\des)}(y,t)\coloneqq\sum_{\pi\in\mathfrak{S}_{n}}y^{\udr(\pi)}t^{\des(\pi)}.
\]
Instead, we will work with 
\[
P_{n}^{(\lpk,\val,\des)}(y,z,t)\coloneqq\sum_{\pi\in\mathfrak{S}_{n}}y^{\lpk(\pi)}z^{\val(\pi)}t^{\des(\pi)},
\]
which is equivalent to $P_{n}^{(\udr,\des)}(y,t)$ in light of Lemma
\ref{l-udr}.
\begin{lem}
\textup{\label{l-udrdesncsf}
\begin{multline*}
(1-t^{2}\mathbf{h}(x)\mathbf{e}(yx))^{-1}(1+t\mathbf{h}(x))=\frac{1}{1-t}\\
+\sum_{n=1}^{\infty}\sum_{L\vDash n}\frac{\begin{array}{l}
t^{\udr(L)}(1+y)^{\udr(L)-1}(1+yt^{2})^{n-1-\des(L)-\val(L)}(y+t^{2})^{\des(L)-\lpk(L)}\\
\qquad\qquad\qquad\qquad\qquad\qquad\times(1+yt)^{1-\lpk(L)+\val(L)}(y+t)^{\lpk(L)-\val(L)}
\end{array}}{(1-t)(1-t^{2})^{n}}x^{n}\mathbf{r}_{L}
\end{multline*}
}
\end{lem}
\begin{proof}
Let us say that a word $w$ on the alphabet $\mathbb{P}\cup\underline{\mathbb{P}}\cup\{|\}$
is a \textit{up-down run word} if $w$ is either:

\begin{itemize}
\item A sequence of subwords of the form $w_{1}|w_{2}|$ where $w_{1}$
is a (possibly empty) weakly increasing word containing only letters
from $\mathbb{P}$ and $w_{2}$ is a (possibly empty) strictly decreasing
word containing only letters from $\underline{\mathbb{P}}$;
\item Or, a sequence of subwords of the form $w_{1}|w_{2}|$ as described
above, but ending with a subword of the form $w_{3}|$, where $w_{3}$
is a (possibly empty) weakly increasing word containing only letters
from $\mathbb{P}$.
\end{itemize}
For example, 
\begin{equation}
12||246678|\underline{98}|4|\underline{321}||\underline{5}|||23|\label{e-udrword}
\end{equation}
is an up-down run word. The left-hand side of the given equation counts
up-down run words where, as before, $t$ is weighting the number of
bars, $y$ is weighting the number of underlined letters, and $x$
is weighting the length of the underlying $\mathbb{P}$-word. We want
show that the right-hand side also counts up-down run words with the
same weights.

Call an up-down run word $w$ \textit{minimal} if it is impossible
to remove bars from $w$ to yield an up-down run word. As before,
every up-down run word can be obtained from only one minimal up-down
run word, which is the only minimal up-down run word on those letters
with the same choice of underlines. For example, the minimal up-down
run word on 
\[
12246678\underline{98}4\underline{3215}23
\]
is 
\[
12246678|\underline{98}|4|\underline{321}||\underline{5}|23|,
\]
which is the unique minimal up-down run word that (\ref{e-udrword})
can be obtained from. We claim that 
\begin{multline}
t(t+yt)^{\udr(L)-1}(1+yt^{2})^{n-1-\des(L)-\val(L)}(y+t^{2})^{\des(L)-\lpk(L)}\\
\times(1+yt)^{1-\lpk(L)+\val(L)}(y+t)^{\lpk(L)-\val(L)}x^{n}\mathbf{r}_{L}\label{e-udrncsf}
\end{multline}
counts nonempty minimal up-down run words with descent composition
$L\vDash n$. Every term in $\mathbf{r}_{L}$ corresponds to a $\mathbb{P}$-word
with descent composition $L$, and we give it a choice of underlines
and insert the necessary bars. Let us take $85432113444889323344513456$
as our working example.

\begin{enumerate}
\item Every up-down run word must end with a bar, so insert a bar at the
end of our word: 
\[
85432123444889323344513456|.
\]
This gives the initial factor of $t$.
\item For each letter corresponding to a left peak or valley (i.e., each
letter that is at the end of an up-down run other than the last one),
we choose whether or not to underline it. For a left peak, if we do
underline it, then we insert a bar immediately before it; otherwise,
insert a bar immediately after it. For a valley, if we do underline
it, then we insert a bar immediately after it; otherwise, insert a
bar immediately before it. This gives the $(t+yt)^{\udr(L)-1}$ factor.
For example, we may have 
\[
|\underline{8}5432|12344488|\underline{9}3\underline{2}|33445|\underline{1}|3456|.
\]
\item For each letter corresponding to a descent that is not a left peak
(i.e., each letter corresponding to a descent and is not the final
letter of an up-down run), choose whether or not to underline it.
If we do not underline the letter, then prepend and append a bar to
it. This gives the $(y+t^{2})^{\des(L)-\lpk(L)}$ factor. For example,
we may have 
\[
|\underline{854}|3|\underline{2}|12344488|\underline{932}|33445|\underline{1}|3456|.
\]
This step eliminates instances of non-underlined letters appearing
in the same segment as an underlined letter, and by adding the bars,
we have a minimal up-down run word corresponding to our current choice
of underlines.
\item For each letter corresponding to an ascent\footnote{We say that $i\in[n-1]$ is an \textit{ascent} of the word $w=w_{1}\cdots w_{n}$
if $w_{i}\leq w_{i+1}$, i.e., if it is not a descent. Ascents are
defined in the analogous way for permutations (with the weak inequality
$\leq$ replaced by the strict inequality $<$ since letters cannot
repeat).} that is not a valley (i.e., each letter corresponding to an ascent
and is not the final letter of an up-down run), choose whether or
not to underline it. If we underline the letter, then also prepend
and append a bar to it. This gives the $(1+yt^{2})^{n-1-\des(L)-\val(L)}$
factor. For example, we may have 
\[
|\underline{854}|3|\underline{2}|12|\underline{3}||\underline{4}|4488|\underline{932}|33|\underline{4}|45|\underline{1}|3456|.
\]
Note that adding the bars is necessary so that the result is a minimal
up-down run word.
\item The only remaining letter of our word that still requires consideration
is the final letter, so the last step is to choose whether or not
to underline it. If the word ends with an increasing run of length
1 (which is equivalent to $\lpk(L)-\val(L)=1$ by Lemma \ref{l-udr})\footnote{Although Lemma \ref{l-udr} was stated for permutations, it also holds
for words.} and we do not underline the final letter, then prepend a bar to it.
If the word ends with an increasing run of length at least 2 (which
is equivalent to $\lpk(L)-\val(L)=0$ by Lemma \ref{l-udr}) and we
underline the final letter, then prepend a bar to it. This gives $(1+yt)^{1-\lpk(L)+\val(L)}(y+t)^{\lpk(L)-\val(L)}$.
For example, we may have 
\[
|\underline{854}|3|\underline{2}|12|\underline{3}||\underline{4}|4488|\underline{932}|33|\underline{4}|45|\underline{1}|345|\underline{6}|.
\]
Again, we have a minimal up-down run word.
\end{enumerate}
We have chosen whether or not to underline each letter in the word,
so (\ref{e-udrncsf}) counts every minimal up-down run word with descent
composition $L\vDash n$. Dividing by $(1-t^{2})^{n}$ allows us to
insert bars in multiples of two at the beginning of the word or between
any two letters; adding them in multiples of two is necessary for
the result to remain an up-down word. However, any number of bars
can be added at the end, hence dividing by $1-t$ as well. This accounts
for all up-down words other than those only consisting of bars, which
are accounted for by the $1/(1-t)$ term.
\end{proof}
\begin{cor}
\textup{\label{c-udrncsf}
\[
{\displaystyle (1-t^{2}\mathbf{h}(x)\mathbf{e}(x))^{-1}(1+t\mathbf{h}(x))=\frac{1}{1-t}+\sum_{n=1}^{\infty}\sum_{L\vDash n}\frac{2^{\udr(L)-1}t^{\udr(L)}(1+t^{2})^{n-\udr(L)}}{(1-t)^{2}(1-t^{2})^{n-1}}x^{n}\mathbf{r}_{L}}
\]
}
\end{cor}
\begin{proof}
Taking Lemma \ref{l-udrdesncsf} and setting $y=1$, we obtain 
\begin{multline*}
(1-t^{2}\mathbf{h}(x)\mathbf{e}(x))^{-1}(1+t\mathbf{h}(x))=\\
\frac{1}{1-t}+\sum_{n=1}^{\infty}\sum_{L\vDash n}\frac{2^{\udr(L)-1}t^{\udr(L)}(1+t^{2})^{n-1-\val(L)-\lpk(L)}(1+t)}{(1-t)(1-t^{2})^{n}}x^{n}\mathbf{r}_{L}.
\end{multline*}
Because $\udr(L)=\lpk(L)+\val(L)+1$ (Lemma \ref{l-udr}), it follows
that 
\begin{align*}
(1-t^{2}\mathbf{h}(x)\mathbf{e}(x))^{-1}(1+t\mathbf{h}(x)) & =\frac{1}{1-t}+\sum_{n=1}^{\infty}\sum_{L\vDash n}\frac{2^{\udr(L)-1}t^{\udr(L)}(1+t^{2})^{n-\udr(L)}(1+t)}{(1-t)(1-t^{2})^{n}}x^{n}\mathbf{r}_{L}\\
 & =\frac{1}{1-t}+\sum_{n=1}^{\infty}\sum_{L\vDash n}\frac{2^{\udr(L)-1}t^{\udr(L)}(1+t^{2})^{n-\udr(L)}}{(1-t)^{2}(1-t^{2})^{n-1}}x^{n}\mathbf{r}_{L}.
\end{align*}
\end{proof}
Since $\udr$ is the only statistic present in the noncommutative
symmetric function identity given in the above corollary, we can use
it to obtain results for the up-down run polynomials $P_{n}^{\udr}(t)$
and their $q$-analogues 
\[
P_{n}^{(\inv,\udr)}(q,t)\coloneqq\sum_{\pi\in\mathfrak{S}_{n}}q^{\inv(\pi)}t^{\udr(\pi)}.
\]

\begin{thm}
\label{t-udr} For $n\geq1$, we have
\begin{equation}
A_{n}(t)=\frac{(1+t^{2})^{n}}{2(1+t)^{n-1}}P_{n}^{\udr}\left(\frac{2t}{1+t^{2}}\right).\label{e-audr}
\end{equation}
Equivalently, 
\begin{equation}
P_{n}^{\udr}(t)=\frac{2(1+v)^{n-1}}{(1+v^{2})^{n}}A_{n}(v)\label{e-udra}
\end{equation}
where $v=\frac{1-\sqrt{1-t^{2}}}{t}$.
\end{thm}
\begin{proof}
Taking Corollary \ref{c-udrncsf}, evaluating at $x=1$, applying
the homomorphism $\Phi$, and rearranging some terms yields 
\begin{align*}
\frac{1}{1-te^{x}}=\frac{1+te^{x}}{1-t^{2}e^{2x}} & =\frac{1}{1-t}+\sum_{n=1}^{\infty}\sum_{\pi\in\mathfrak{S}_{n}}\frac{(1+t^{2})^{n}}{2(1-t)^{2}(1-t^{2})^{n-1}}\left(\frac{2t}{1+t^{2}}\right)^{\udr(\pi)}\frac{x^{n}}{n!}.
\end{align*}
Multiplying both sides by $1-t$ and then replacing $x$ by $(1-t)x$
yields 
\begin{align*}
\frac{1-t}{1-te^{(1-t)x}} & =1+\sum_{n=1}^{\infty}\sum_{\pi\in\mathfrak{S}_{n}}\frac{(1+t^{2})^{n}}{2(1+t)^{n-1}}\left(\frac{2t}{1+t^{2}}\right)^{\udr(\pi)}\frac{x^{n}}{n!}.
\end{align*}
The left-hand side is precisely the exponential generating function
for the Eulerian polynomials, so equating the coefficients of $x^{n}/n!$
gives (\ref{e-audr}). Then (\ref{e-udra}) can be obtained by making
the substitution $v=2t/(1+t)^{2}$ and solving for $t$.
\end{proof}
\begin{thm}
\label{t-udrq} We have \textup{
\begin{align}
\frac{(1-t)(1+t\exp_{q}(x))}{1-t^{2}\exp_{q}(x)\Exp_{q}(x)} & =1+\frac{1+t}{2}\sum_{n=1}^{\infty}\left(\frac{1+t^{2}}{1-t^{2}}\right)^{n}P_{n}^{(\inv,\udr)}\left(q,\frac{2t}{1+t^{2}}\right)\frac{x^{n}}{[n]_{q}!}.\label{e-qinvudr1}
\end{align}
}Equivalently, \textup{
\begin{equation}
\sum_{n=1}^{\infty}P_{n}^{(\inv,\udr)}(q,t)\frac{x^{n}}{[n]_{q}!}=\frac{2}{1+v}\left(\frac{(1-v)\left(1+v\exp_{q}\left(\frac{1-v^{2}}{1+v^{2}}x\right)\right)}{1-v^{2}\exp_{q}\left(\frac{1-v^{2}}{1+v^{2}}x\right)\Exp_{q}\left(\frac{1-v^{2}}{1+v^{2}}x\right)}-1\right)\label{e-qinvudr2}
\end{equation}
}where $v=\frac{1-\sqrt{1-t^{2}}}{t}.$
\end{thm}
\begin{proof}
Apply the homomorphism $\Phi_{q}$ to Corollary \ref{c-udrncsf} evaluated
at $x=1$; then multiplying both sides by $1-t$ yields (\ref{e-qinvudr1}).

Next, replace $x$ by $x(1-t^{2})/(1+t^{2})$ to get 
\begin{align*}
\frac{(1-t)\left(1+t\exp_{q}\left(\frac{1-t^{2}}{1+t^{2}}x\right)\right)}{1-t^{2}\exp_{q}\left(\frac{1-t^{2}}{1+t^{2}}x\right)\Exp_{q}\left(\frac{1-t^{2}}{1+t^{2}}\right)} & =1+\frac{1+t}{2}\sum_{n=1}^{\infty}P_{n}^{(\inv,\udr)}\left(q,\frac{2t}{1+t^{2}}\right)\frac{x^{n}}{[n]_{q}!}.
\end{align*}
Then rearranging some terms and making the same substitution as in
the proof of Theorem \ref{t-udr} yields (\ref{e-qinvudr2}).
\end{proof}
Recall that when setting $y=1$ in Lemma \ref{l-udrdesncsf}, all
instances of the statistics $\lpk$ and $\val$ either cancel out
or reduce to $\udr$. This is not possible in the general form of
Lemma \ref{l-udrdesncsf}, so we cannot directly work with the polynomials
$P_{n}^{(\udr,\des)}(y,t)=\sum_{\pi\in\mathfrak{S}_{n}}y^{\udr(\pi)}t^{\des(\pi)}$.
Since the statistics $(\udr,\des)$ and $(\lpk,\val,\des)$ are equivalent,
we will instead give results for the polynomials $P_{n}^{(\lpk,\val,\des)}(y,z,t)=\sum_{\pi\in\mathfrak{S}_{n}}y^{\lpk(\pi)}z^{\val(\pi)}t^{\des(\pi)}$
and their $q$-analogues 
\[
P_{n}^{(\inv,\lpk,\val,\des)}(q,y,z,t)\coloneqq\sum_{\pi\in\mathfrak{S}_{n}}q^{\inv(\pi)}y^{\lpk(\pi)}z^{\val(\pi)}t^{\des(\pi)}.
\]

\begin{thm}
\label{t-lpkvaldes} For $n\geq1$, we have
\begin{multline*}
\frac{(1+y)^{n}}{t}A_{n}(t^{2})+\sum_{k=0}^{n}{n \choose k}(1+y)^{k}(1-t^{2})^{n-k}A_{k}(t^{2})=(1+yt)(1+t)(1+yt^{2})^{n-1}\\
\times P_{n}^{(\lpk,\val,\des)}\left(\frac{t(1+y)(y+t)}{(y+t^{2})(1+yt)},\frac{t(1+y)(1+yt)}{(1+yt^{2})(y+t)},\frac{y+t^{2}}{1+yt^{2}}\right).
\end{multline*}
\end{thm}

\begin{thm}
\label{t-lpkvaldesq} \textup{
\begin{multline*}
\frac{(1-t)(1+t\exp_{q}(x))}{1-t^{2}\exp_{q}(x)\Exp_{q}(yx)}=1+t(1+yt)\sum_{n=1}^{\infty}\frac{(1+yt^{2})^{n-1}}{(1-t^{2})^{n}}\\
\times P_{n}^{(\inv,\lpk,\val,\des)}\left(q,\frac{t(1+y)(y+t)}{(y+t^{2})(1+yt)},\frac{t(1+y)(1+yt)}{(1+yt^{2})(y+t)},\frac{y+t^{2}}{1+yt^{2}}\right)\frac{x^{n}}{[n]_{q}!}
\end{multline*}
}
\end{thm}
We omit the proofs of the above two theorems as they follow in essentially
the same way as the proofs of Theorems \ref{t-udr} and \ref{t-udrq},
except that we would use Lemma \ref{l-udrdesncsf} rather than its
specialization (Corollary \ref{c-udrncsf}). Unlike in these theorems,
however, it is not possible to invert the identities to give an explicit
expression for $P_{n}^{(\lpk,\val,\des)}(y,z,t)$ or for the $q$-exponential
generating function for $P_{n}^{(\inv,\lpk,\val,\des)}(q,y,z,t)$.

Finally, we relate the $(\lpk,\val,\des)$ polynomials to the refined
flag descent polynomials $F_{n}(y,t)=\sum_{\pi\in\mathfrak{B}_{n}}y^{\neg(\pi)}t^{\fdes(\pi)}$,
which specializes to a relation between the $\udr$ polynomials and
the flag descent polynomials $F_{n}(t)=\sum_{\pi\in\mathfrak{B}_{n}}t^{\fdes(\pi)}$.
\begin{thm}
\label{t-lpkvaldesb} For $n\geq1$, we have 
\[
F_{n}(y,t)=(1+yt)(1+yt^{2})^{n-1}P_{n}^{(\lpk,\val,\des)}\left(\frac{t(1+y)(y+t)}{(y+t^{2})(1+yt)},\frac{t(1+y)(1+yt)}{(1+yt^{2})(y+t)},\frac{y+t^{2}}{1+yt^{2}}\right).
\]
\end{thm}
\begin{proof}
Follows immediately from Theorems \ref{t-fnan} and \ref{t-lpkvaldes}.
\end{proof}
We give a combinatorial proof of this result in Subsection 5.3.
\begin{cor}
\label{c-fudr} For $n\geq1$, we have 
\begin{equation}
F_{n}(t)=\frac{(1+t)(1+t^{2})^{n}}{2t}P_{n}^{\udr}\left(\frac{2t}{1+t^{2}}\right).\label{e-udrf1}
\end{equation}
Equivalently, 
\begin{equation}
P_{n}^{\udr}(t)=\frac{2v}{(1+v)(1+v^{2})^{n}}F_{n}(v)\label{e-udrf2}
\end{equation}
where $v=\frac{1-\sqrt{1-t^{2}}}{t}$.
\end{cor}
\begin{proof}
We obtain (\ref{e-udrf1}) by taking the preceding theorem, setting
$y=1$, and rearranging a few terms. Then (\ref{e-udrf2}) is obtained
from (\ref{e-udrf1}) by the same substitution as before.
\end{proof}

\subsection{Two remarks: the inverse major index and alternating analogues}

We end this section with two important remarks.

First, our formulas for $q$-analogues of descent statistic polynomials
are also valid for the ``inverse major index''. For a permutation
$\pi\in\mathfrak{S}$, the \textit{major index} $\maj(\pi)$ is defined
to be the sum of its descents, and the \textit{inverse major index}
$\imaj(\pi)$ is the major index of its inverse when considered as
an element of the symmetric group. For example, take $\pi=85712643$,
whose inverse is $\pi^{-1}=45872631$. Since $\Des(\pi)=\{1,3,6,7\}$
and $\Des(\pi^{-1})=\{3,4,6,7\}$, the major index of $\pi$ is 
\[
\maj(\pi)=1+3+6+7=17
\]
whereas the inverse major index of $\pi$ is 
\[
\imaj(\pi)=\maj(\pi^{-1})=3+4+6+7=20.
\]

A remarkable result by Foata and Sch\"utzenberger \cite{Foata1978}
states that the inversion number $\inv$ and inverse major index $\imaj$
are equidistributed over descent classes. That is, for any $S\subseteq[n-1]$,
\[
\sum_{\substack{\pi\in\mathfrak{S}_{n}\\
\Des(\pi)=S
}
}q^{\inv(\pi)}=\sum_{\substack{\pi\in\mathfrak{S}_{n}\\
\Des(\pi)=S
}
}q^{\imaj(\pi)};
\]
this is equivalent to saying that the polynomial $\beta_{q}(L)$---defined
in Subsection 2.2---counting $n$-permutations with descent composition
$L$ by inversion number also counts these same permutations by inverse
major index. It follows that Theorems \ref{t-pkdesq}, \ref{t-lpkdesq},
\ref{t-udrq}, and \ref{t-lpkvaldesq} and Corollaries \ref{c-pkq}
and \ref{c-lpkq} can be restated for the inverse major index, that
is, by replacing every instance of $\inv$ with $\imaj$.

Second, we mention that there exist ``alternating analogues'' for
nearly every concept in this paper relating to descents, as well as
for several results in this section. Given a permutation $\pi\in\mathfrak{S}_{n}$,
we say that $i\in[n-1]$ is an \textit{alternating descent} of $\pi$
if $i$ is an odd descent or an even ascent, and an \textit{alternating
run} of $\pi$ is a maximal consecutive subsequence of $\pi$ containing
no alternating descents. Then the alternating descent set, alternating
descent composition, and the alternating descent number $\altdes$
can all be defined in the obvious way. The distribution of $\altdes$
over $\mathfrak{S}_{n}$ is given by the polynomial 
\[
\hat{A}_{n}(t)\coloneqq\sum_{\pi\in\mathfrak{S}_{n}}t^{\altdes(\pi)+1},
\]
the alternating analogue of the $n$th Eulerian polynomial. The exponential
generating function for the alternating Eulerian polynomials is 
\[
\sum_{n=0}^{\infty}\hat{A}_{n}(t)\frac{x^{n}}{n!}=\frac{1-t}{1-t(\sec((1-t)x)+\tan((1-t)x))},
\]
which is the exponential generating function for the ordinary Eulerian
polynomials with the exponential function $e^{x}$ replaced by $\sec x+\tan x$.

There is in fact an alternating analogue for every descent statistic.
For example, since $i$ is a peak of $\pi$ precisely when $i-1$
is an ascent and $i$ is a descent, we could define $i$ to be an
alternating peak if $i-1$ is an ``alternating ascent'' and $i$ is
an alternating descent, that is, if $\pi_{i-1}>\pi_{i}>\pi_{i+1}$
and $i$ is odd or if $\pi_{i-1}<\pi_{i}<\pi_{i+1}$ and $i$ is even.

Alternating descents were introduced by Chebikin \cite{chebikin}
and later brought into our noncommutative symmetric function framework
by Gessel and the present author \cite{Gessel2014}. In \cite{Gessel2014},
the authors defined a third homomorphism $\hat{\Phi}:\mathbf{Sym}\rightarrow F[[x]]$
by $\hat{\Phi}(\mathbf{h}_{n})=E_{n}x^{n}/n!$, where $E_{n}$ is
the $n$th Euler number defined by $\sum_{n=0}^{\infty}E_{n}x^{n}/n!=\sec x+\tan x$,
and showed that $\hat{\Phi}(\mathbf{r}_{L})=\hat{\beta}(L)x^{n}/n!$,
where $\hat{\beta}(L)$ is the number of $n$-permutations with alternating
descent composition $L$. Moreover, it is not hard to show that $\hat{\Phi}(\mathbf{h}(1))=\hat{\Phi}(\mathbf{e}(1))=\sec x+\tan x$,
and one can use these facts to obtain alternating analogues of Theorems
\ref{t-pkdes}, \ref{t-lpkdes}, \ref{t-udr}, and \ref{t-lpkvaldes}
that express alternating analogues of descent statistic polynomials
in terms of the alternating Eulerian polynomials.

\section{Proofs via group actions}

\subsection{The modified Foata--Strehl action and $(\protect\pk,\protect\des)$
polynomials}

In Theorem \ref{t-pkdes}, we established a connection between the
$n$th Eulerian polynomial---the polynomial encoding the distribution
of the descent number $\des$ over $\mathfrak{S}_{n}$---and the
polynomial encoding the joint distribution of the peak number $\pk$
and $\des$ over $\mathfrak{S}_{n}$. We can give a generalization
of this result which establishes the same connection between polynomials
encoding the distribution of these statistics over subsets of $\mathfrak{S}_{n}$
that are invariant under a certain group action. We begin by defining
this group action and later use it to prove the aforementioned generalization
of Theorem \ref{t-pkdes}.

In what follows, it will be convenient to alter the definitions of
descent, peak, and valley to refer to the letter and not the index;
that is, given a permutation $\pi=\pi_{1}\pi_{2}\cdots\pi_{n}$, we
call $\pi_{i}$ (rather than $i$) a \textit{descent} of $\pi$ if
$\pi_{i}>\pi_{i+1}$, a \textit{peak} of $\pi$ if $\pi_{i-1}<\pi_{i}>\pi_{i+1}$,
and a \textit{valley} of $\pi$ if $\pi_{i-1}>\pi_{i}<\pi_{i+1}$.\footnote{Note that the statistics $\des$, $\pk$, and $\val$ are the same
regardless of which definition is used.} In addition, $\pi_{i}$ is called a \textit{double ascent} of $\pi$
if $\pi_{i-1}<\pi_{i}<\pi_{i+1}$ and is called a \textit{double descent}
of $\pi$ if $\pi_{i-1}>\pi_{i}>\pi_{i+1}$. Denote the number of
double ascents of $\pi$ by $\dasc(\pi)$ and the number of double
descents of $\pi$ by $\ddes(\pi)$. We will be working with peaks,
valleys, double ascents, and double descents of the word $\infty\pi\infty$,
where $\infty$ is a letter defined to be greater than any integer
in the usual ordering. For example, if $\pi$ is a permutation with
$\pi_{1}<\pi_{2}$, then $\pi_{1}$ is a valley of $\infty\pi\infty$
even though it is not a valley of $\pi$.

Let $\pi$ be an $n$-permutation and let $x\in[n]$. We may write
$\pi=w_{1}w_{2}xw_{4}w_{5}$ where $w_{2}$ is the maximal consecutive
subword immediately to the left of $x$ whose letters are all smaller
than $x$, and $w_{4}$ is the maximal consecutive subword immediately
to the right of $x$ whose letters are all smaller than $x$; this
is called the \textit{$x$-factorization} of $\pi$. For example,
if $\pi=467125839$ and $x=5$, then $\pi$ is the concatenation of
$w_{1}=467$, $w_{2}=12$, $x=5$, the empty word $w_{4}$, and $w_{5}=839$.

Define $\varphi_{x}:\mathfrak{S}_{n}\rightarrow\mathfrak{S}_{n}$
by $\varphi_{x}(\pi)=w_{1}w_{4}xw_{2}w_{5}$. It is easy to see that
$\varphi_{x}$ is an involution, and that for all $x,y\in[n]$, the
involutions $\varphi_{x}$ and $\varphi_{y}$ commute with each other.
These involutions were defined by Foata and Strehl \cite{Foata1974}
in the context of the original Foata--Strehl action, but this was
modified by Br\"and\'en \cite{Braenden2008} in the following way.
If $x\in[n]$, then define $\varphi_{x}^{\prime}:\mathfrak{S}_{n}\rightarrow\mathfrak{S}_{n}$
by
\[
\varphi_{x}^{\prime}(\pi)=\begin{cases}
\varphi_{x}(\pi), & \mbox{if }x\mbox{ is a double ascent or double descent of \ensuremath{\infty\pi\infty},}\\
\pi, & \mbox{if }x\mbox{ is a peak or valley of \ensuremath{\infty\pi\infty}.}
\end{cases}
\]
Equivalently, $\varphi_{x}^{\prime}(\pi)=\varphi_{x}(\pi)$ if exactly
one of $w_{2}$ and $w_{4}$ is nonempty, and $\varphi_{x}^{\prime}(\pi)=\pi$
otherwise. Returning to the example $\pi=467125839$ and $x=5$, we
have $\varphi_{x}^{\prime}(\pi)=467512839$ because $x=5$ is a double
ascent. However, if $x=8$, then $\varphi_{x}^{\prime}(\pi)=\pi$
since $x=8$ is a peak. Note that if $x$ is a double ascent of $\pi$,
then it is a double descent of $\varphi_{x}^{\prime}(\pi)$, and vice
versa.

For any subset $S\subseteq[n]$, we define $\varphi_{S}^{\prime}:\mathfrak{S}_{n}\rightarrow\mathfrak{S}_{n}$
by $\varphi_{S}^{\prime}=\prod_{x\in S}\varphi_{x}^{\prime}$. These
$\varphi_{S}^{\prime}$ are also involutions, and for all $S,T\subseteq[n]$,
we still have that $\varphi_{S}^{\prime}$ commutes with $\varphi_{T}^{\prime}$.
The group $\mathbb{Z}_{2}^{n}$ acts on $\mathfrak{S}_{n}$ via the
involutions $\varphi_{S}^{\prime}$ over all $S\subseteq[n]$; this
group action is called the \textit{modified Foata}--\textit{Strehl
action}, abbreviated \textit{MFS-action}. The MFS-action can also
be characterized in terms of ``valley hopping'' (see \cite[p. 173]{Petersen2013}
or \cite[Section 4.2]{Petersen2015}), and can be defined for decreasing
binary trees (see \cite[p. 516]{Braenden2008}).

The involutions $\varphi_{S}^{\prime}$ were considered earlier by
Shapiro, Woan, and Getu \cite{Shapiro1983}, but they did not seem
to be aware of the connection to the work of Foata and Strehl. The
MFS-action was used much more extensively by Br\"and\'en, whose
results in \cite{Braenden2008} motivated much of the work in this
section. 

Given a set $\Pi$ of $n$-permutations, let 
\[
A(\Pi;t)\coloneqq\sum_{\pi\in\Pi}t^{\des(\pi)+1}
\]
be the $n$th descent polynomial restricted to permutations in $\Pi$
and let 
\[
P^{(\pk,\des)}(\Pi;y,t)\coloneqq\sum_{\pi\in\Pi}y^{\pk(\pi)+1}t^{\des(\pi)+1}
\]
be the $(\pk,\des)$ polynomial for permutations in $\Pi$.
\begin{thm}
\label{t-pkdesmfs} Suppose that $\Pi\subseteq\mathfrak{S}_{n}$ for
$n\geq1$ is invariant under the MFS-action. Then 
\begin{equation}
A(\Pi;t)=\left(\frac{1+yt}{1+y}\right)^{n+1}P^{(\pk,\des)}\left(\Pi;\frac{(1+y)^{2}t}{(y+t)(1+yt)},\frac{y+t}{1+yt}\right).\label{e-apkdesmfs}
\end{equation}
Equivalently, 
\begin{equation}
P^{(\pk,\des)}(\Pi;y,t)=\left(\frac{1+u}{1+uv}\right)^{n+1}A(\Pi;v)\label{e-pkdesamfs}
\end{equation}
where $u=\frac{1+t^{2}-2yt-(1-t)\sqrt{(1+t)^{2}-4yt}}{2(1-y)t}$ and
$v=\frac{(1+t)^{2}-2yt-(1+t)\sqrt{(1+t)^{2}-4yt}}{2yt}.$
\end{thm}
Note that taking $\Pi=\mathfrak{S}_{n}$ yields Theorem \ref{t-pkdes}.
\begin{proof}
Let $\sigma$ be an $n$-permutation and let $\Orb(\sigma)=\{\,\varphi_{S}^{\prime}(\sigma)\mid S\subseteq[n]\,\}$
be the orbit of $\sigma$ under the MFS-action. Let $\check{\pi}=\infty\pi\infty$
and $\check{\sigma}=\infty\sigma\infty$, and for a subset $S\subseteq[n]$,
let $\pi_{S}=\varphi_{S}^{\prime}(\pi)$ and $\check{\pi}_{S}=\infty\varphi_{S}^{\prime}(\pi)\infty$.
First we wish to prove the identity 
\begin{equation}
\Big(\sum_{\pi\in\Orb(\sigma)}t^{\des(\pi)}\Big)(1+y)^{\dasc(\check{\sigma})+\ddes(\check{\sigma})}=\sum_{\pi\in\Orb(\sigma)}(1+yt)^{\dasc(\check{\pi})}(y+t)^{\ddes(\check{\pi})}t^{\pk(\check{\pi})},\label{e-mfs}
\end{equation}
which we do by showing that the two sides of the equation encode the
same objects.

We begin with the left-hand side. Each summand in the factor $\sum_{\pi\in\Orb(\sigma)}t^{\des(\pi)}$
corresponds to a permutation in the orbit of $\sigma$ weighted by
its descent number. Each summand in the factor $(1+y)^{\dasc(\check{\sigma})+\ddes(\check{\sigma})}$
corresponds to marking a subset of double ascents and double descents
of $\check{\sigma}$. Thus, the left-hand side counts permutations
in $\Orb(\sigma)$ where $t$ is weighting the descent number and
$y$ is weighting the number of marked letters.

Now, let us examine the right-hand side of (\ref{e-mfs}). Each term
on the right-hand side corresponds to taking a permutation $\pi\in\Orb(\sigma)$,
choosing a subset $S$ of double ascents and double descents of $\check{\pi}$,
applying $\varphi_{S}^{\prime}$ to $\pi$, marking the letters of
$S$ in $\pi_{S}$---which are all double ascents or double descents
of $\check{\pi}_{S}$---and weighting the marked letters by $y$
and the descents of $\pi_{S}$ by $t$. The $(1+yt)^{\dasc(\check{\pi})}$
factor corresponds to selecting the double ascents, and the $(y+t)^{\ddes(\check{\pi})}$
factor corresponds to selecting the double descents. The only remaining
descents of $\pi_{S}$ are the peaks of $\check{\pi}_{S}$, which
are precisely the peaks of $\check{\pi}$ since the MFS-action does
not affect peaks; this contributes the factor of $t^{\pk(\check{\pi})}$.
In summary, both sides of the equation count permutations in the orbit
of $\sigma$ with a marked subset $S$ of letters by the same weights,
but on the right-hand side, we are applying the involution $\varphi_{S}^{\prime}$
to each $\pi\in\Orb(\sigma)$ before doing the counting.

Next, observe the following:

\begin{itemize}
\item The number of peaks of $\check{\pi}$ is equal to $\pk(\pi)$. More
specifically, the peaks of $\pi$ are precisely the peaks of $\check{\pi}$,
since neither $\pi_{1}$ nor $\pi_{n}$ can be peaks of $\check{\pi}=\infty\pi\infty$.
\item The number of double descents of $\check{\pi}$ is equal to $\des(\pi)-\pk(\pi)$,
since the set of descents of $\pi$ is equal to the set of double
descents and peaks of $\check{\pi}$.
\item The number of double ascents of $\check{\pi}$ is equal to $n-\pk(\pi)-\des(\pi)-1$.
To see why this is true, first observe that the number of valleys
of $\check{\pi}$ is $\pk(\pi)+1$. Then, the double ascents of $\check{\pi}$
are the letters which are not peaks, valleys, or double descents of
$\check{\pi}$. Thus, 
\begin{align*}
\dasc(\check{\pi}) & =n-(\pk(\check{\pi})+\val(\check{\pi})+\ddes(\check{\pi}))\\
 & =n-(\pk(\pi)+\pk(\pi)+1+\des(\pi)-\pk(\pi))\\
 & =n-\pk(\pi)-\des(\pi)-1.
\end{align*}
\item The total number of double ascents and double descents of $\check{\sigma}$
is equal to $n-2\pk(\sigma)-1$, which follows from the previous two
points.
\end{itemize}
Therefore, we have the equation
\[
\Big(\sum_{\pi\in\Orb(\sigma)}t^{\des(\pi)}\Big)(1+y)^{n-2\pk(\sigma)-1}=\sum_{\pi\in\Orb(\sigma)}(1+yt)^{n-\pk(\pi)-\des(\pi)-1}(y+t)^{\des(\pi)-\pk(\pi)}t^{\pk(\pi)},
\]
and dividing both sides by $(1+y)^{n-2\pk(\sigma)-1}=(1+y)^{n-2\pk(\pi)-1}$
gives us 
\[
\sum_{\pi\in\Orb(\sigma)}t^{\des(\pi)}=\sum_{\pi\in\Orb(\sigma)}\frac{(1+yt)^{n-\pk(\pi)-\des(\pi)-1}(y+t)^{\des(\pi)-\pk(\pi)}t^{\pk(\pi)}}{(1+y)^{n-2\pk(\pi)-1}}.
\]
Then, summing over all orbits corresponding to permutations in $\Pi$
yields 
\[
\sum_{\pi\in\Pi}t^{\des(\pi)}=\sum_{\pi\in\Pi}\frac{(1+yt)^{n-\pk(\pi)-\des(\pi)-1}(y+t)^{\des(\pi)-\pk(\pi)}t^{\pk(\pi)}}{(1+y)^{n-2\pk(\pi)-1}},
\]
and multiplying both sides by $t$ yields 
\begin{align*}
A(\Pi;t) & =\sum_{\pi\in\Pi}\frac{(1+yt)^{n-\pk(\pi)-\des(\pi)-1}(y+t)^{\des(\pi)-\pk(\pi)}t^{\pk(\pi)+1}}{(1+y)^{n-2\pk(\pi)-1}}\\
 & =\left(\frac{1+yt}{1+y}\right)^{n+1}P^{(\pk,\des)}\left(\Pi;\frac{(1+y)^{2}t}{(y+t)(1+yt)},\frac{y+t}{1+yt}\right),
\end{align*}
which is (\ref{e-apkdesmfs}). We obtain (\ref{e-pkdesamfs}) by performing
the same substitution as in the proof of Theorem \ref{t-pkdes}.
\end{proof}
Setting $y=1$ in (\ref{e-apkdesmfs}) yields a main result of Br\"and\'en
\cite[Corollary 3.2]{Braenden2008}, which we state below for completeness.
For $\Pi\subseteq\mathfrak{S}_{n}$, let 
\[
P^{\pk}(\Pi;t)\coloneqq\sum_{\pi\in\Pi}t^{\pk(\pi)+1}.
\]

\begin{cor}
Suppose that $\Pi\subseteq\mathfrak{S}_{n}$ for $n\geq1$ is invariant
under the MFS-action. Then 
\[
A(\Pi;t)=\left(\frac{1+t}{2}\right)^{n+1}P^{\pk}\left(\Pi;\frac{4t}{(1+t)^{2}}\right).
\]
Equivalently, 
\[
P^{\pk}(\Pi;t)=\left(\frac{2}{1+v}\right)^{n+1}A(\Pi;v)
\]
where $v=\frac{2}{t}(1-\sqrt{1-t})-1$.
\end{cor}
Besides $\Pi=\mathfrak{S}_{n}$, other subsets $\Pi$ of $\mathfrak{S}_{n}$
invariant under the MFS-action include the set of $r$-stack-sortable
$n$-permutations $\mathfrak{S}_{n}^{r}$ (see \cite[Corollary 4.2]{Braenden2008})
for any $r\geq1$. Of particular interest is the set of 1-stack-sortable
$n$-permutations $\mathfrak{S}_{n}^{1}$, which Knuth \cite{Knuth1973}
showed to be precisely the set $\Av_{n}(231)$ of 231-avoiding $n$-permutations:
$n$-permutations such that there exist no $i<j<k$ for which $\pi_{k}<\pi_{i}<\pi_{j}$.
The descent polynomial for 231-avoiding $n$-permutations is known
to be the $n$th Narayana polynomial 
\[
N_{n}(t)\coloneqq\sum_{\pi\in\Av_{n}(231)}t^{\des(\pi)+1}=\frac{1}{n}\sum_{k=1}^{n}{n \choose k}{n \choose k-1}t^{k}
\]
for $n\geq1$. Hence, we have the following corollary.
\begin{cor}
\label{c-pkdes231} For $n\geq1$, we have \textup{
\begin{equation}
N_{n}(t)=\left(\frac{1+yt}{1+y}\right)^{n+1}P^{(\pk,\des)}\left(\Av_{n}(231);\frac{(1+y)^{2}t}{(y+t)(1+yt)},\frac{y+t}{1+yt}\right).\label{e-npkdes231}
\end{equation}
}Equivalently, \textup{
\begin{align*}
P^{(\pk,\des)}(\Av_{n}(231);y,t) & =\left(\frac{1+u}{1+uv}\right)^{n+1}N_{n}(v)
\end{align*}
}where $u=\frac{1+t^{2}-2yt-(1-t)\sqrt{(1+t)^{2}-4yt}}{2(1-y)t}$
and $v=\frac{(1+t)^{2}-2yt-(1+t)\sqrt{(1+t)^{2}-4yt}}{2yt}.$ 
\end{cor}
In the next subsection, we examine the polynomials $P^{(\pk,\des)}(\Av_{n}(231);y,t)$
in greater detail and interpret this corollary in terms of binary
trees and Dyck paths. Setting $y=1$ in (\ref{e-npkdes231}) gives
an identity relating the $n$th Narayana polynomial to the peak polynomial
for $231$-avoiding $n$-permutations, which is equivalent to Equation
4.7 of \cite[p. 77]{Petersen2015}.

We also state the result for 2-stack-sortable permutations, which
can again be characterized in terms of pattern avoidance: the set
of 2-stack-sortable $n$-permutations $\mathfrak{S}_{n}^{2}$ is equal
to the set $\Av_{n}(2341,3\bar{5}241)$ of $n$-permutations avoiding
the pattern 2341 and the ``barred pattern'' $3\bar{5}241$ (see \cite[Theorem 4.2.18]{West1990}).\footnote{A permutation is said to avoid the ``barred pattern'' $3\bar{5}241$
if no 3241-pattern occurs without there being a larger letter between
the letters corresponding to 3 and 2. Note that $3\bar{5}241$ is
not a signed permutation. For more general definitions on notions
of pattern avoidance in permutations, see \cite{Kitaev2011} for a
comprehensive reference.} The descent polynomial for these $n$-permutations is given by the
following formula, which was proven by Jacquard and Schaeffer \cite{Jacquard1998}
using a bijection with nonseparable planar maps: 
\[
A(\mathfrak{S}_{n}^{2};t)=\sum_{k=1}^{n}\frac{(n+k-1)!(2n-k)!}{k!(n-k+1)!(2k-1)!(2n-2k+1)!}t^{k}.
\]

\begin{cor}
\label{c-pkdes2ss} For $n\geq1$, we have \textup{
\begin{align*}
A(\mathfrak{S}_{n}^{2};t) & =\left(\frac{1+yt}{1+y}\right)^{n+1}P^{(\pk,\des)}\left(\mathfrak{S}_{n}^{2};\frac{(1+y)^{2}t}{(y+t)(1+yt)},\frac{y+t}{1+yt}\right)
\end{align*}
}Equivalently, \textup{
\begin{align*}
P^{(\pk,\des)}(\mathfrak{S}_{n}^{2};y,t) & =\left(\frac{1+u}{1+uv}\right)^{n+1}A(\mathfrak{S}_{n}^{2};v)
\end{align*}
}where $u=\frac{1+t^{2}-2yt-(1-t)\sqrt{(1+t)^{2}-4yt}}{2(1-y)t}$
and $v=\frac{(1+t)^{2}-2yt-(1+t)\sqrt{(1+t)^{2}-4yt}}{2yt}.$
\end{cor}
Furthermore, there are various statistics $\st$ that are constant
on any orbit of the MFS-action. These include the number of occurrences
of the generalized permutation pattern 23-1 and the number of occurrences
of 13-2 (see \cite[Theorem 5.1]{Braenden2008}) as well as the number
of ``admissible inversions'' (see \cite[Lemma 7]{Lin2015}). Thus
if we refine the polynomials $A(\Pi;t)$ and $P^{(\pk,\des)}(\Pi;y,t)$
by defining 
\[
A^{\st}(\Pi;t,w)\coloneqq\sum_{\pi\in\Pi}t^{\des(\pi)+1}w^{\st(\pi)}
\]
and 
\[
P^{(\pk,\des,\st)}(\Pi;y,t,w)\coloneqq\sum_{\pi\in\Pi}y^{\pk(\pi)+1}t^{\des(\pi)+1}w^{\st(\pi)},
\]
then we have the following result.
\begin{cor}
\label{c-pkdesst} Suppose that $\Pi\subseteq\mathfrak{S}_{n}$ for
$n\geq1$ is invariant under the MFS-action, and let $\st$ be a permutation
statistic that is constant on any orbit of the MFS-action. Then 
\[
A^{\st}(\Pi;t,w)=\left(\frac{1+yt}{1+y}\right)^{n+1}P^{(\pk,\des,\st)}\left(\Pi;\frac{(1+y)^{2}t}{(y+t)(1+yt)},\frac{y+t}{1+yt},w\right).
\]
Equivalently, 
\[
P^{(\pk,\des,\st)}(\Pi;y,t,w)=\left(\frac{1+u}{1+uv}\right)^{n+1}A^{\st}(\Pi;v,w)
\]
where $u=\frac{1+t^{2}-2yt-(1-t)\sqrt{(1+t)^{2}-4yt}}{2(1-y)t}$ and
$v=\frac{(1+t)^{2}-2yt-(1+t)\sqrt{(1+t)^{2}-4yt}}{2yt}.$ 
\end{cor}

\subsection{231-avoiding permutations, binary trees, and Dyck paths}

Corollary \ref{c-pkdes231} shows that the polynomial 
\[
P^{(\pk,\des)}(\Av_{n}(231);y,t)=\sum_{\pi\in\Av_{n}(231)}y^{\pk(\pi)+1}t^{\des(\pi)+1}
\]
can be expressed in terms of the $n$th Narayana polynomial $N_{n}(t)$.
However, it is worth noting that there is a simple formula for the
polynomials $P^{(\pk,\des)}(\Av_{n}(231);y,t)$ and their coefficients.
\begin{thm}
\label{t-pkdes231f} For $n\geq1$, we have 
\begin{equation}
P^{(\pk,\des)}(\Av_{n}(231);y,t)=\sum_{k=0}^{\left\lfloor (n-1)/2\right\rfloor }\frac{1}{k+1}{2k \choose k}{n-1 \choose 2k}y^{k+1}t^{k+1}(1+t)^{n-2k-1},\label{e-pkdes231f1}
\end{equation}
and the number of $231$-avoiding $n$-permutations with exactly $k$
peaks and $j$ descents is 
\begin{equation}
\frac{1}{k+1}{2k \choose k}{n-1 \choose 2k}{n-2k-1 \choose j-k}.\label{e-pkdes231f2}
\end{equation}
\end{thm}
By setting $t=1$ in (\ref{e-pkdes231f1}), we see that 
\[
\frac{1}{k+1}{2k \choose k}{n-1 \choose 2k}2^{n-2k-1}
\]
is the number of $231$-avoiding $n$-permutations with exactly $k$
peaks; this is Corollary 4.5 of Br\"and\'en \cite{Braenden2008}.
\begin{proof}
Define $G=G(y,t,x)$ to be the ordinary generating function for the
polynomials $P^{(\pk,\des)}(\Av_{n}(231);y,t)/y$ for $n\geq1$. Then
$G$ satisfies the functional equation 
\[
G=x(yG^{2}+tG+G+t);
\]
this can be seen using the bijection with binary trees given later
in this subsection. Then a straightforward application of Lagrange
inversion yields (\ref{e-pkdes231f1}), and (\ref{e-pkdes231f2})
is easily obtained from (\ref{e-pkdes231f1}) using the binomial theorem.
\end{proof}
We recall the definitions of two families of objects closely related
to $231$-avoiding permutations. A \textit{binary tree}\footnote{These are sometimes also called \textit{binary plane trees} or \textit{planar
binary trees} in the literature.} is a rooted tree $T$ satisfying the following two conditions:
\begin{enumerate}
\item Each node (i.e., vertex) of $T$ has 0, 1, or 2 children.
\item Each child of each node is distinguished as a \textit{left child}
or a \textit{right child}.
\end{enumerate}
Let ${\cal T}_{n}$ be the set of binary trees with $n$ nodes. For
example, below is a binary tree with 6 nodes.

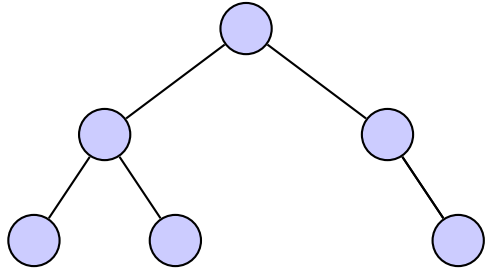
\begin{center}
\begin{tikzpicture}[scale=0.47,auto,   thick,main node/.style={circle,fill=blue!20,draw,font=\sffamily,minimum size=1.65em},leaf/.style={circle,fill=red!20,draw,font=\sffamily,minimum size=0.5em}]

\node[main node] (1) at (10,10) {};
\node[main node] (2) at (6,7) {};
\node[main node] (3) at (14,7) {};
\node[main node] (4) at (4,4) {};
\node[main node] (5) at (8,4) {};
\node[main node] (6) at (16,4) {};

\foreach \from/\to in {1/2,1/3,2/4,2/5,3/6,3/6}
\draw (\from) -- (\to);

\end{tikzpicture}
\end{center}

We will work with two particular statistics defined on binary trees.
If $T$ is a binary tree, then let $\nlc(T)$ be the number of nodes
of $T$ with no left children, and let $\tc(T)$ be the number of
nodes of $T$ with two children. For example, if $T$ is the binary
tree given above, then $\nlc(T)=4$ and $\tc(T)=2$.

A \textit{Dyck path} of \textit{semilength} $n$ is a lattice path
in the plane from $(0,0)$ to $(2n,0)$, consisting of up steps $(1,1)$
and down steps $(-1,1)$, and never going below the $x$-axis. Let
${\cal D}_{n}$ be the set of Dyck paths of semilength $n$. A Dyck
path can be encoded by a word on the alphabet $\{U,D\}$ in which
$U$ corresponds to an up step and $D$ a down step; this is called
the \textit{Dyck word} corresponding to the Dyck path. For example,
below is a Dyck path with semilength 6 corresponding to the Dyck word
$UUDUUUDDDDUD$.

\noindent 
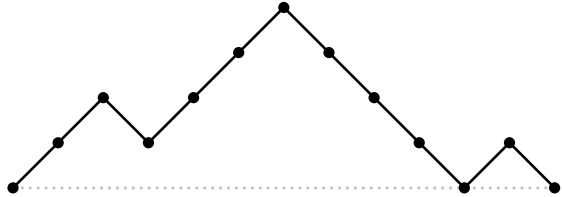
\begin{center}
\begin{tikzpicture}[scale=0.6] 
\draw [line width=0] (4,2); 
\draw[pathcolorlight] (0,0) -- (12,0); 
\drawlinedots{0,1,2,3,4,5,6,7,8,9,10,11,12}{0,1,2,1,2,3,4,3,2,1,0,1,0} 
\end{tikzpicture} \qquad{}
\par\end{center}

We define two statistics on Dyck paths: the number of ``peaks'' and
the number of ``hooks''. Define a \textit{peak} of a Dyck path $\mu$
to be an occurrence of the subword $UD$ in the corresponding Dyck
word of $\mu$, and a \textit{hook} of $\mu$ to be an occurrence
of the subword $DDU$. Let $\pk(\mu)$ be the number of peaks of $\mu$
and $\hk(\mu)$ the number of hooks of $\mu$. For example, if $\mu$
is the Dyck path given above, then $\pk(\mu)=3$ and $\hk(\mu)=1$.

It is well known that binary trees with $n$ nodes, Dyck paths with
semilength $n$, and 231-avoiding $n$-permutations are all counted
by the $n$th Catalan number $C_{n}$ given by 
\[
C_{n}\coloneqq\frac{1}{n+1}{2n \choose n}.
\]
A number of bijections exist between these three sets of objects.
Below we will use two particular bijections which behave nicely with
respect to the statistics defined above. The first bijection $\Theta:\Av_{n}(231)\rightarrow{\cal T}_{n}$
satisfies $\des(\pi)+1=\nlc(\Theta(\pi))$ and $\pk(\pi)=\tc(\Theta(\pi))$.
The second bijection $\Psi:\Av_{n}(231)\rightarrow{\cal D}_{n}$ satisfies
$\des(\pi)+1=\pk(\Psi(\pi))$ and $\pk(\pi)=\hk(\Psi(\pi))$. These
bijections will allow us to give interpretations of Corollary \ref{c-pkdes231}
and Theorem \ref{t-pkdes231f} for binary trees and Dyck paths.

A \textit{decreasing binary tree} $T$ with $n$ nodes is a binary
tree with nodes labeled using distinct letters from $[n]$ such that
if $i,j$ are nodes of $T$ with $j$ a child of $i$, then $i>j$.\footnote{We are identifying each node with its label.}
Let us briefly recall the classical recursive bijection $\tilde{\Theta}$
between permutations\footnote{Here we are using the term ``permutation'' in a broader sense, referring
to a linear ordering of any set of distinct integers.} and decreasing binary trees. If $\pi$ is the empty permutation,
then let $\tilde{\Theta}(\pi)=\emptyset$. Otherwise, let $n$ be
the largest letter of $\pi$, and factor $\pi$ as $\pi=\sigma n\tau$
where $\sigma$ is the (possibly empty) subsequence of $\pi$ consisting
of all letters to the left of $n$ and $\tau$ is the (possibly empty)
subsequence consisting of all letters to the right of $n$. Then $\tilde{\Theta}(\pi)$
is defined by letting $n$ be the root of $\tilde{\Theta}(\pi)$,
attaching $\tilde{\Theta}(\sigma)$ as a left subtree, and attaching
$\tilde{\Theta}(\tau)$ as a right subtree. It is clear that $\tilde{\Theta}$
gives a bijection between $\mathfrak{S}_{n}$ and ${\cal T}_{n}$
for all $n$.
\begin{lem}
\label{l-pbt} Let $\pi\in\mathfrak{S}_{n}$ with $n\geq1$. Then
$\des(\pi)+1=\nlc(\tilde{\Theta}(\pi))$ and $\pk(\pi)=\tc(\tilde{\Theta}(\pi))$.
\end{lem}
\begin{proof}
Fix $\pi\in\mathfrak{S}_{n}$ with $n\geq1$, and write $\pi=\sigma n\tau$.
It is easy to verify that 
\[
\des(\pi)+1=\begin{cases}
\des(\sigma)+1+\des(\tau)+1, & \mbox{if }\left|\sigma\right|\geq1\,\mbox{ and }\left|\tau\right|\geq1,\\
\des(\sigma)+1, & \mbox{if }\left|\sigma\right|\geq1\,\mbox{ and }\left|\tau\right|=0,\\
\des(\tau)+1+1, & \mbox{if }\left|\sigma\right|=0\,\mbox{ and }\left|\tau\right|\geq1,\\
1, & \mbox{if }\left|\sigma\right|=\left|\tau\right|=0,
\end{cases}
\]
and 
\[
\nlc(\tilde{\Theta}(\pi))=\begin{cases}
\nlc(\tilde{\Theta}(\sigma))+\nlc(\tilde{\Theta}(\tau)), & \mbox{if }\left|\sigma\right|\geq1\,\mbox{ and }\left|\tau\right|\geq1,\\
\nlc(\tilde{\Theta}(\sigma)), & \mbox{if }\left|\sigma\right|\geq1\,\mbox{ and }\left|\tau\right|=0,\\
\nlc(\tilde{\Theta}(\tau))+1, & \mbox{if }\left|\sigma\right|=0\,\mbox{ and }\left|\tau\right|\geq1,\\
1, & \mbox{if }\left|\sigma\right|=\left|\tau\right|=0,
\end{cases}
\]
so it follows from induction that $\des(\pi)+1=\nlc(\tilde{\Theta}(\pi))$
for all nonempty permutations $\pi$. Similarly, 
\[
\pk(\pi)=\begin{cases}
\pk(\sigma)+\pk(\tau)+1, & \mbox{if }\left|\sigma\right|\geq1\,\mbox{ and }\left|\tau\right|\geq1,\\
\pk(\sigma)+\pk(\tau), & \mbox{otherwise},
\end{cases}
\]
and 
\[
\tc(\tilde{\Theta}(\pi))=\begin{cases}
\tc(\tilde{\Theta}(\sigma))+\tc(\tilde{\Theta}(\tau))+1, & \mbox{if }\left|\sigma\right|\geq1\,\mbox{ and }\left|\tau\right|\geq1,\\
\tc(\tilde{\Theta}(\sigma))+\tc(\tilde{\Theta}(\tau)), & \mbox{otherwise},
\end{cases}
\]
so we have $\pk(\pi)=\tc(\tilde{\Theta}(\pi))$ as well.
\end{proof}
For example, the permutation $\pi=132495876$ is mapped by $\tilde{\Theta}$
to the decreasing binary tree below.

\begin{center}
\begin{tikzpicture}[scale=0.47,auto,   thick,main node/.style={circle,fill=blue!20,draw,font=\sffamily,minimum size=1em},leaf/.style={circle,fill=red!20,draw,font=\sffamily,minimum size=0.5em}]

\node[main node] (9) at (10,10) {9};
\node[main node] (4) at (6,7) {4};
\node[main node] (8) at (14,7) {8};
\node[main node] (3) at (4,4) {3};
\node[main node] (5) at (12,4) {5};
\node[main node] (7) at (16,4) {7};
\node[main node] (1) at (3,1) {1};
\node[main node] (2) at (5,1) {2};
\node[main node] (6) at (17,1) {6};

\foreach \from/\to in {9/4,9/8,4/3,8/5,8/7,3/1,3/2,7/6}
\draw (\from) -- (\to);

\end{tikzpicture}
\end{center}Observe that $\des(\pi)+1=\nlc(\tilde{\Theta}(\pi))=5$ and $\pk(\pi)=\tc(\tilde{\Theta}(\pi))=3$.
To define the map $\Theta:\Av_{n}(231)\rightarrow{\cal T}_{n}$, first
notice that the permutation $\pi=132495876$ is 231-avoiding, and
that performing a post-order traversal of the corresponding tree $\tilde{\Theta}(\pi)$
yields the increasing permutation 123456789. Indeed, it is not hard
to verify that a permutation $\pi\in\mathfrak{S}_{n}$ is 231-avoiding
if and only if a post-order traversal of $\tilde{\Theta}(\pi)$ yields
$12\cdots n$. Thus, if we define $\Theta(\pi)$ for $\pi\in\Av_{n}(231)$
to be the binary tree obtained by taking $\tilde{\Theta}(\pi)$ and
removing the labels, then we can recover $\tilde{\Theta}(\pi)$ from
$\Theta(\pi)$ by labeling the nodes via post-order traversal; it
follows that we have a well-defined bijection $\Theta:\Av_{n}(231)\rightarrow{\cal T}_{n}$
satisfying the properties from Lemma \ref{l-pbt}. Therefore, the
$n$th Narayana polynomial $N_{n}(t)$ gives the distribution of binary
trees with $n$ nodes by the number of nodes without left children,
and by defining 
\[
T_{n}^{(\tc,\nlc)}(y,t)\coloneqq\sum_{T\in{\cal T}_{n}}y^{\tc(T)+1}t^{\nlc(T)},
\]
we have proven the following interpretation of Corollary \ref{c-pkdes231}
and Theorem \ref{t-pkdes231f}.
\begin{cor}
\label{c-tcnlc} For any $n\geq1$, we have \textup{
\[
N_{n}(t)=\left(\frac{1+yt}{1+y}\right)^{n+1}T_{n}^{(\tc,\nlc)}\left(\frac{(1+y)^{2}t}{(y+t)(1+yt)},\frac{y+t}{1+yt}\right).
\]
}Equivalently, \textup{
\begin{align*}
T_{n}^{(\tc,\nlc)}(y,t) & =\left(\frac{1+u}{1+uv}\right)^{n+1}N_{n}(v)
\end{align*}
}where $u=\frac{1+t^{2}-2yt-(1-t)\sqrt{(1+t)^{2}-4yt}}{2(1-y)t}$
and $v=\frac{(1+t)^{2}-2yt-(1+t)\sqrt{(1+t)^{2}-4yt}}{2yt}.$ Furthermore,
\[
T_{n}^{(\tc,\nlc)}(y,t)=\sum_{k=0}^{\left\lfloor (n-1)/2\right\rfloor }\frac{1}{k+1}{2k \choose k}{n-1 \choose 2k}y^{k+1}t^{k+1}(1+t)^{n-2k-1}
\]
and the number of binary trees with $n$ nodes, $k$ nodes with two
children, and $j$ nodes with no left children is 
\[
\frac{1}{k+1}{2k \choose k}{n-1 \choose 2k}{n-2k-1 \choose j-k-1}.
\]
\end{cor}
Now we describe the bijection $\Psi:\Av_{n}(231)\rightarrow{\cal D}_{n}$,
which was defined by Stump in \cite{Stump2009}. Let $\pi\in\Av_{n}(231)$
with $\Comp(\pi)=(L_{1},L_{2},\dots,L_{k})\vDash n$ and $\Comp(\pi^{-1})=(K_{1},K_{2},\dots,K_{k})\vDash n$.\footnote{In defining this map, Stump proves that $\des(\pi)=\des(\pi^{-1})$
if $\pi$ avoids $\sigma$ for all $\sigma\in\{132,231,312,213\}$.
Thus, if $\pi\in\Av_{n}(231)$, then the descent compositions of $\pi$
and $\pi^{-1}$ have the same number of parts.} Then $\Psi(\pi)$ is defined to be the Dyck path with corresponding
Dyck word $U^{K_{1}}D^{L_{1}}U^{K_{2}}D^{L_{2}}\cdots U^{K_{k}}D^{L_{k}}$,
in which $U^{m}$ represents 
\[
\underbrace{UU\cdots U}_{m\:\mathrm{times}}
\]
and similarly for $D^{m}$. It is not obvious that this word actually
corresponds to a Dyck path, or that $\Psi$ is a bijection even if
it is well-defined; we refer the reader to Stump's paper \cite{Stump2009}
for proofs.

For example, take the permutation $\pi=219438567$, which has inverse
$\pi^{-1}=215478963$. Then $\Comp(\pi)=(1,2,1,2,3)$ and $\Comp(\pi^{-1})=(1,2,4,1,1)$.
Then $\Psi(\pi)$ is the Dyck path corresponding to the Dyck word
$UDUUDDUUUUDUDDUDDD$, pictured below.

\noindent \begin{center}
\begin{tikzpicture}[scale=0.6] 
\draw [line width=0] (4,2); 
\draw[pathcolorlight] (0,0) -- (18,0); 
\drawlinedots{0,1,2,3,4,5,6,7,8,9,10,11,12,13,14,15,16,17,18}{0,1,0,1,2,1,0,1,2,3,4,3,4,3,2,3,2,1,0}
\end{tikzpicture} \qquad{}
\par\end{center}
\begin{lem}
\label{l-dyck} Let $\pi\in\mathfrak{S}_{n}$ with $n\geq1$. Then
$\des(\pi)+1=\pk(\Psi(\pi))$ and $\pk(\pi)=\hk(\Psi(\pi))$.
\end{lem}
\begin{proof}
It is obvious from the definition of $\Psi$ that $\des(\pi)+1$---the
number of parts of $\Comp(\pi)$---is equal to the number of peaks
of $\Psi(\pi)$. 

To show that $\pk(\pi)=\hk(\Psi(\pi))$, let $\pi\in\Av_{n}(231)$
with descent composition $L=(L_{1},L_{2},\dots,L_{k})$. Since a peak
of $\pi$ corresponds to an occurrence of an ascent followed by a
descent, it follows that $\pk(\pi)$ is equal to the number of non-final
parts of $L$ of length at least 2 (that is, not counting $L_{k}$
if $L_{k}\geq2$). Each part $L_{i}$ of $L$ contributes to $\Psi(\pi)$
a sequence of $L_{i}$ down steps, so if $L_{i}\geq2$, then we have
a sequence of at least 2 down steps. This sequence of down steps is
followed by an up step if $i\neq k$ (i.e., if $L_{i}$ is not the
final part of $L$), so each non-final part of $L$ of length at least
2 gives rise to a hook. Moreover, it is easy to see that every hook
of $\Psi(\pi)$ arises in this way, so $\pk(\pi)=\hk(\Psi(\pi))$.
\end{proof}
For example, taking $\pi=219438567$ as above, we have $\des(\pi)+1=\pk(\Psi(\pi))=5$
and $\pk(\pi)=\hk(\Psi(\pi))=2$. It follows from Lemma \ref{l-dyck}
that the Narayana polynomials $N_{n}(t)$ also give the distribution
of Dyck paths with semilength $n$ by number of peaks, and if we define
\[
D_{n}^{(\hk,\pk)}(y,t)\coloneqq\sum_{\mu\in{\cal D}_{n}}y^{\hk(\mu)+1}t^{\pk(\mu)},
\]
we obtain yet another interpretation of Corollary \ref{c-pkdes231}
and Theorem \ref{t-pkdes231f}.
\begin{cor}
\label{c-hkpk} For any $n\geq1$, we have \textup{
\[
N_{n}(t)=\left(\frac{1+yt}{1+y}\right)^{n+1}D_{n}^{(\hk,\pk)}\left(\frac{(1+y)^{2}t}{(y+t)(1+yt)},\frac{y+t}{1+yt}\right).
\]
}Equivalently, \textup{
\begin{align*}
D_{n}^{(\hk,\pk)}(y,t) & =\left(\frac{1+u}{1+uv}\right)^{n+1}N_{n}(v)
\end{align*}
}where $u=\frac{1+t^{2}-2yt-(1-t)\sqrt{(1+t)^{2}-4yt}}{2(1-y)t}$
and $v=\frac{(1+t)^{2}-2yt-(1+t)\sqrt{(1+t)^{2}-4yt}}{2yt}.$ Furthermore,
\[
D_{n}^{(\hk,\pk)}(y,t)=\sum_{k=0}^{\left\lfloor (n-1)/2\right\rfloor }\frac{1}{k+1}{2k \choose k}{n-1 \choose 2k}y^{k+1}t^{k+1}(1+t)^{n-2k-1}
\]
and the number of Dyck paths of semilength $n$ with $k$ hooks and
$j$ peaks is 
\[
\frac{1}{k+1}{2k \choose k}{n-1 \choose 2k}{n-2k-1 \choose j-k-1}.
\]
\end{cor}

\subsection{The Petersen action and $(\protect\lpk,\protect\des)$ and $(\protect\lpk,\protect\val,\protect\des)$
polynomials}

Just as we showed that Theorem \ref{t-pkdes} can be generalized for
subsets of $\mathfrak{S}_{n}$ that are invariant under the modified
Foata--Strehl action, we will generalize Theorems \ref{t-lpkdesb}
and \ref{t-lpkvaldesb} for subsets of $\mathfrak{B}_{n}$ that are
invariant under a different $\mathbb{Z}_{2}^{n}$-action described
by Petersen in \cite[Section 13.2]{Petersen2015}. For this, we revert
to our original definitions of descent, peak, and valley as referring
to the index rather than the letter. Similarly, we redefine double
ascent and double descent to refer to the index as well.

For a permutation $\pi\in\mathfrak{S}_{n}$, let $\mathfrak{B}(\pi)$
be the set of signed permutations that can be obtained by choosing
any subset of the letters in $\pi$ and replacing them with their
negatives. For example, if $\pi=213$, then $\mathfrak{B}(\pi)=\{213,21\bar{3},2\bar{1}3,2\bar{1}\bar{3},\bar{2}13,\bar{2}1\bar{3},\bar{2}\bar{1}3,\bar{2}\bar{1}\bar{3}\}$.
By thinking of $\pi$ as a signed permutation, we see that $\pi$
is the unique representative of $\mathfrak{B}(\pi)$ without negative
signs, and $\mathfrak{B}(\pi)$ is the orbit of $\pi$ under the $\mathbb{Z}_{2}^{n}$-action
on $\mathfrak{B}_{n}$ given by taking a subset of letters and reversing
their signs.

Before proving our main results, we establish an important lemma that
will allow us to determine the distribution of $(\neg,\des_{B})$
or $(\neg,\fdes)$ over the orbit of any given $\pi\in\mathfrak{S}_{n}$
simply by knowing the number of peaks, valleys, double ascents, and
double descents of the permutation $\acute{\pi}=0\pi\infty$ of $\{0,\infty\}\cup[n]$.
\begin{lem}
\label{l-bdes} Let $\pi\in\mathfrak{S}_{n}$, $\acute{\pi}=0\pi\infty$,
and $\sigma\in\mathfrak{B}(\pi)$.

\begin{itemize}
\item If $i$ is a peak of $\acute{\pi}$, then $i-1$ is a descent of $\sigma$
if and only if $\sigma_{i}<0$, and $i$ is a descent of $\sigma$
if and only if $\sigma_{i}>0$.
\item If $i$ is a double ascent of $\acute{\pi}$, then $i-1$ is a descent
of $\sigma$ if and only if $\sigma_{i}<0$.
\item If $i$ is a double descent of $\acute{\pi}$, then $i$ is a descent
of $\sigma$ if and only if $\sigma_{i}>0$.
\end{itemize}
Moreover, every descent of $\sigma$ is accounted for exactly once
in the above.
\end{lem}
\begin{proof}
Suppose that $i$ is a peak of $\acute{\pi}$, so that $\pi_{i-1}<\pi_{i}>\pi_{i+1}$.\footnote{This is assuming that $i\neq1$, but the case $i=1$ follows similarly.}
If $\sigma_{i-1}>\sigma_{i}$, then $\sigma_{i-1}\leq\pi_{i-1}<\pi_{i}=\pm\sigma_{i}$,
which forces $\sigma_{i}=-\pi_{i}$ and thus $\sigma_{i}<0$. If $\sigma_{i}>\sigma_{i+1}$,
then $\sigma_{i+1}\leq\pi_{i+1}<\pi_{i}=\pm\sigma_{i}$, which forces
$\sigma_{i}=\pi_{i}$ and thus $\sigma_{i}>0$. On the other hand,
if $\sigma_{i}<0$, then $\sigma_{i}=-\pi_{i}<-\pi_{i-1}\leq\sigma_{i-1}$
and $\sigma_{i}=-\pi_{i}<-\pi_{i+1}\leq\sigma_{i+1}$, so $i-1$ is
a descent and $i$ is an ascent. The statements for double ascents
and double descents follow using similar reasoning.

We now verify that every descent of $\sigma$ is accounted for exactly
once. Fix an $i$ between 0 and $n-1$. If $i$ is a peak or double
descent of $\acute{\pi}$, then $i+1$ is either a double descent
or valley of $\acute{\pi}$. Then the sign of $\sigma_{i}$ determines
whether or not $i$ is a descent of $\sigma$, whereas the sign of
$\sigma_{i+1}$ has no effect on whether or not $i$ is a descent
of $\sigma$. If $i$ is a valley or double ascent of $\acute{\pi}$
or if $i=0$, then $i+1$ is either a double ascent or peak of $\acute{\pi}$.
Then the sign of $\sigma_{i+1}$ determines whether or not $i$ is
a descent of $\sigma$, whereas the sign of $\sigma_{i}$ has no effect
on whether or not $i$ is a descent of $\sigma$. This shows that
every descent is accounted for exactly once in the above.
\end{proof}
Given a set $\Pi\subseteq\mathfrak{S}_{n}$, let $\mathfrak{B}(\Pi)\coloneqq\bigcup_{\pi\in\Pi}\mathfrak{B}(\pi)$.
Also, define the polynomials 
\[
B(\Pi;y,t)\coloneqq\sum_{\pi\in\mathfrak{B}(\Pi)}y^{\neg(\pi)}t^{\des_{B}(\pi)}
\]
and 
\[
P^{(\lpk,\des)}(\Pi;y,t)\coloneqq\sum_{\pi\in\Pi}y^{\lpk(\pi)}t^{\des(\pi)}.
\]

\begin{thm}
\label{t-palpkdes} Let $\Pi\subseteq\mathfrak{S}_{n}$ for $n\geq0$.
Then 
\begin{equation}
B(\Pi;y,t)=(1+yt)^{n}P^{(\lpk,\des)}\left(\Pi;\frac{(1+y)^{2}t}{(y+t)(1+yt)},\frac{y+t}{1+yt}\right).\label{e-blpkdespa}
\end{equation}
Equivalently, 
\begin{equation}
P^{(\lpk,\des)}(\Pi;y,t)=\frac{B(\Pi;u,v)}{(1+uv)^{n}}\label{e-lpkdesbpa}
\end{equation}
where $u=\frac{1+t^{2}-2yt-(1-t)\sqrt{(1+t)^{2}-4yt}}{2(1-y)t}$ and
$v=\frac{(1+t)^{2}-2yt-(1+t)\sqrt{(1+t)^{2}-4yt}}{2yt}.$ 
\end{thm}
By taking $\Pi=\mathfrak{S}_{n}$, we recover Theorem \ref{t-lpkdesb}.
\begin{proof}
For a fixed $\pi\in\Pi$, let $\acute{\pi}=0\pi\infty$. Then it follows
from Lemma \ref{l-bdes} that
\[
\sum_{\sigma\in\mathfrak{B}(\pi)}y^{\neg(\sigma)}t^{\des_{B}(\sigma)}=((1+y)t)^{\pk(\acute{\pi})}(1+y)^{\val(\acute{\pi})}(1+yt)^{\dasc(\acute{\pi})}(y+t)^{\ddes(\acute{\pi})}.
\]

Now, observe the following:

\begin{itemize}
\item The number of peaks of $\acute{\pi}$ is equal to $\lpk(\pi)$. Indeed,
every peak of $\pi$ is a peak of $\acute{\pi}$, but if $\pi$ begins
with a descent, then that contributes an additional peak to $\acute{\pi}$.
\item The number of double descents of $\acute{\pi}$ is equal to $\des(\pi)-\lpk(\pi)$,
since the set of descents of $\pi$ is equal to the set of double
descents and peaks of $\acute{\pi}$.
\item The number of valleys of $\acute{\pi}$ is equal to $\lpk(\pi)$,
but this is harder to see. Suppose that $n-1$ is a descent of $\pi$.
Then $\val(\acute{\pi})=\val(\pi)+1$. We know from Lemma \ref{l-udr}
that, in this case, $\lpk(\pi)=\val(\pi)+1$, so indeed $\val(\acute{\pi})=\lpk(\pi)$.
If $n-1$ is not a descent of $\pi$, then $\val(\acute{\pi})=\val(\pi)$,
and since we know that $\lpk(\pi)=\val(\pi)$ from Lemma \ref{l-udr},
the claim follows.
\item The number of double ascents of $\acute{\pi}$ is equal to $n-\lpk(\pi)-\des(\pi)$,
since
\begin{align*}
\dasc(\acute{\pi}) & =n-(\pk(\acute{\pi})+\val(\acute{\pi})+\ddes(\acute{\pi}))\\
 & =n-(\lpk(\pi)+\lpk(\pi)+\des(\pi)-\lpk(\pi)\\
 & =n-\lpk(\pi)-\des(\pi).
\end{align*}
\end{itemize}
Hence, we have 
\begin{align*}
\sum_{\sigma\in\mathfrak{B}(\pi)}y^{\neg(\pi)}t^{\des_{B}(\sigma)} & =((1+y)^{2}t)^{\lpk(\pi)}(y+t)^{\des(\pi)-\lpk(\pi)}(1+yt)^{n-\lpk(\pi)-\des(\pi)}\\
 & =(1+yt)^{n}\left(\frac{(1+y)^{2}t}{(y+t)(1+yt)}\right)^{\lpk(\pi)}\left(\frac{y+t}{1+yt}\right)^{\des(\pi)}.
\end{align*}
Summing over all $\pi\in\Pi$ yields (\ref{e-blpkdespa}), and (\ref{e-lpkdesbpa})
is obtained via the same substitutions as before.
\end{proof}
By setting $y=1$ in (\ref{e-blpkdespa}), we obtain the following
specialization. For $\Pi\subseteq\mathfrak{S}_{n}$, define 
\[
B(\Pi;t)\coloneqq\sum_{\pi\in\mathfrak{B}(\Pi)}t^{\des_{B}(\pi)}
\]
and 
\[
P^{\lpk}(\Pi;t)\coloneqq\sum_{\pi\in\Pi}t^{\lpk(\pi)}.
\]

\begin{cor}
\label{c-palpk} Let $\Pi\subseteq\mathfrak{S}_{n}$ for $n\geq0$.
Then 
\[
B(\Pi;t)=(1+t)^{n}P^{\lpk}\left(\Pi;\frac{4t}{(1+t)^{2}}\right).
\]
Equivalently, 
\[
P^{\lpk}(\Pi;t)=\frac{B(\Pi;v)}{(1+v)^{n}}
\]
where $v=\frac{2}{t}(1-\sqrt{1-t})-1$.
\end{cor}
Similarly to Corollary \ref{c-pkdesst}, we can refine Theorem \ref{t-palpkdes}
by statistics that are constant on any orbit of the described action.
These are the statistics $\st$ on $\mathfrak{B}$ such that when
$\sigma,\tau\in\mathfrak{B}(\pi)$---that is, if $\sigma$ and $\tau$
reduce to the same $\pi\in\mathfrak{S}$ when their negative signs
are removed---we have $\st(\sigma)=\st(\tau)$, or in other words,
statistics $\st$ that depend only on the underlying unsigned permutation.
Clearly, these are the permutation statistics on $\mathfrak{S}$.\footnote{More precisely, each such statistic induces a permutation statistic
on $\mathfrak{S}$, and conversely, each permutation statistic on
$\mathfrak{S}$ induces a statistic on $\mathfrak{B}$ that is constant
on every orbit.} Thus, for a statistic $\st$ on $\mathfrak{S}$ and a signed permutation
$\sigma\in\mathfrak{B}$, we let $\st(\sigma)$ mean the value of
$\st$ on the unsigned permutation obtained by removing all negative
signs from $\sigma$. 

For a permutation statistic $\st$ on $\mathfrak{S}$, define

\[
B^{\st}(\Pi;y,t,w)\coloneqq\sum_{\pi\in\mathfrak{B}(\Pi)}y^{\neg(\pi)}t^{\des_{B}(\pi)}w^{\st(\pi)}
\]
and 
\[
P^{(\lpk,\des,\st)}(\Pi;y,t,w)\coloneqq\sum_{\pi\in\Pi}y^{\lpk(\pi)}t^{\des(\pi)}w^{\st(\pi)}.
\]

\begin{cor}
Let $\st$ be a permutation statistic on $\mathfrak{S}$ and let $\Pi\subseteq\mathfrak{S}_{n}$
for $n\geq0$. Then 
\[
B^{\st}(\Pi;y,t,w)=(1+yt)^{n}P^{(\lpk,\des,\st)}\left(\Pi;\frac{(1+y)^{2}t}{(y+t)(1+yt)},\frac{y+t}{1+yt},w\right).
\]
Equivalently, 
\[
P^{(\lpk,\des,\st)}(\Pi;y,t,w)=\frac{B^{\st}(\Pi;u,v,w)}{(1+uv)^{n}}
\]
where $u=\frac{1+t^{2}-2yt-(1-t)\sqrt{(1+t)^{2}-4yt}}{2(1-y)t}$ and
$v=\frac{(1+t)^{2}-2yt-(1+t)\sqrt{(1+t)^{2}-4yt}}{2yt}.$ 
\end{cor}
To conclude this section, we state the corresponding results for the statistics $(\lpk,\val,\des)$
and $(\neg,\fdes)$. For $\Pi\subseteq\mathfrak{S}_{n}$, define 
\[
F(\Pi;y,t)\coloneqq\sum_{\pi\in\mathfrak{B}(\Pi)}y^{\neg(\pi)}t^{\fdes(\pi)}
\]
and

\[
P^{(\lpk,\val,\des)}(\Pi;y,z,t)\coloneqq\sum_{\pi\in\Pi}y^{\lpk(\pi)}z^{\val(\pi)}t^{\des(\pi)}.
\]

Recall that the \textit{reverse complement} $\pi^{rc}$ of a permutation
$\pi=\pi_{1}\pi_{2}\cdots\pi_{n}\in\mathfrak{S}_{n}$ is given by
\[
\pi^{rc}\coloneqq(n+1-\pi_{n})(n+1-\pi_{n-1})\cdots(n+1-\pi_{1}).
\]
That is, we first reverse the order of the letters in $\pi$, and
then replace the $i$th smallest letter with the $i$th largest letter
for all $1\leq i\leq n$. For example, if $\pi=1723465$, then the
reverse complement of $\pi$ is given by $\pi^{rc}=3245617$. Also,
let the reverse complement $\Pi^{rc}$ of a set of permutations $\Pi\subseteq\mathfrak{S}_{n}$
be the set of their reverse complements: $\Pi^{rc}\coloneqq\{\,\pi^{rc}\mid\pi\in\Pi\,\}.$

In the proof of the next theorem, we will be working with the reverse
complement of the permutation $\acute{\pi}=0\pi\infty$ of $\{0,\infty\}\cup[n]$
where $\pi\in\mathfrak{S}_{n}$, which we define to be $\acute{\pi}^{rc}\coloneqq0\pi^{rc}\infty$. 
\begin{thm}
\label{t-palpkvaldes} Let $\Pi\subseteq\mathfrak{S}_{n}$ for $n\geq1$.
Then 
\begin{multline*}
F(\Pi;y,t)=(1+yt)(1+yt^{2})^{n-1}\\
\times P^{(\lpk,\val,\des)}\left(\Pi^{rc};\frac{t(1+y)(y+t)}{(y+t^{2})(1+yt)},\frac{t(1+y)(1+yt)}{(1+yt^{2})(y+t)},\frac{y+t^{2}}{1+yt^{2}}\right).
\end{multline*}
\end{thm}
By taking $\Pi=\mathfrak{S}_{n}$, we recover Theorem \ref{t-lpkvaldesb}.
\begin{proof}
For a fixed $\pi\in\Pi$ and $\sigma\in\mathfrak{B}(\pi)$, let $\acute{\pi}=0\pi\infty$.
Then it follows from Lemma \ref{l-bdes} that
\[
\sum_{\sigma\in\mathfrak{B}(\pi)}y^{\neg(\sigma)}t^{\fdes(\sigma)}=c_{1}(\pi)\cdots c_{n}(\pi)
\]
where 
\[
c_{1}(\pi)=\begin{cases}
1+yt, & \mbox{if }\pi_{1}<\pi_{2},\\
t(y+t), & \mbox{if }\pi_{1}>\pi_{2},
\end{cases}
\]
and for $i>1$, 
\[
c_{i}(\pi)=\begin{cases}
t^{2}(1+y), & \mbox{if }i\mbox{ is a peak of }\acute{\pi},\\
1+y, & \mbox{if }i\mbox{ is a valley of }\acute{\pi},\\
1+yt^{2}, & \mbox{if }i\mbox{ is a double ascent of }\acute{\pi},\\
y+t^{2}, & \mbox{if }i\mbox{ is a double descent of }\acute{\pi}.
\end{cases}
\]
By taking the reverse complement of $\acute{\pi}$, we have that 
\[
\sum_{\sigma\in\mathfrak{B}(\pi)}y^{\neg(\sigma)}t^{\fdes(\sigma)}=d_{1}(\pi)\cdots d_{n}(\pi)
\]
where
\[
d_{n}(\pi)=\begin{cases}
1+yt, & \mbox{if }\pi_{n-1}^{rc}<\pi_{n}^{rc},\\
t(y+t), & \mbox{if }\pi_{n-1}^{rc}>\pi_{n}^{rc},
\end{cases}
\]
and for $i<n$,
\[
d_{i}(\pi)=\begin{cases}
t^{2}(1+y), & \mbox{if }i\mbox{ is a valley of }\acute{\pi}^{rc},\\
1+y, & \mbox{if }i\mbox{ is a peak of }\acute{\pi}^{rc},\\
1+yt^{2}, & \mbox{if }i\mbox{ is a double ascent of }\acute{\pi}^{rc},\\
y+t^{2}, & \mbox{if }i\mbox{ is a double descent of }\acute{\pi}^{rc}.
\end{cases}
\]
It follows from Lemma \ref{l-udr} that $d_{n}(\pi)=(t(y+t))^{\lpk(\pi^{rc})-\val(\pi^{rc})}(1+yt)^{1+\val(\pi^{rc})-\lpk(\pi^{rc})}$;
the $1+yt$ factor disappears if $n$ is a valley of $\acute{\pi}^{rc}$,
and the $t(y+t)$ factor disappears if instead $n$ is a double ascent
of $\acute{\pi}^{rc}$, which are the only two possibilities.

Next, we determine the contribution of the $d_{i}(\pi)$ for $i<n$
where $i$ is a peak or valley of $\acute{\pi}^{rc}$. The remaining
number of valleys of $\acute{\pi}^{rc}$---that is, not including
$n$---is equal to $\val(\pi^{rc})$, and the number of peaks of
$\acute{\pi}^{rc}$ is equal to $\lpk(\pi^{rc})$. When $n$ is a
valley of $\acute{\pi}^{rc}$ and thus $\lpk(\pi^{rc})=\val(\pi^{rc})+1$,
the total contribution from the peaks, valleys, and the final letter
is 
\begin{multline*}
t^{\lpk(\pi^{rc})+\val(\pi^{rc})-1}(1+y)^{\lpk(\pi^{rc})+\val(\pi^{rc})}(t(y+t))^{\lpk(\pi^{rc})-\val(\pi^{rc})}(1+yt)^{1+\val(\pi^{rc})-\lpk(\pi^{rc})}=\\
t^{\lpk(\pi^{rc})+\val(\pi^{rc})}(1+y)^{\lpk(\pi^{rc})+\val(\pi^{rc})}(y+t)^{\lpk(\pi^{rc})-\val(\pi^{rc})}(1+yt)^{1+\val(\pi^{rc})-\lpk(\pi^{rc})}.
\end{multline*}
On the other hand, when $n$ is a double ascent of $\acute{\pi}^{rc}$
and thus $\lpk(\pi^{rc})=\val(\pi^{rc})$, the contribution is 
\begin{multline*}
t^{\lpk(\pi^{rc})+\val(\pi^{rc})}(1+y)^{\lpk(\pi^{rc})+\val(\pi^{rc})}(t(y+t))^{\lpk(\pi^{rc})-\val(\pi^{rc})}(1+yt)^{1+\val(\pi^{rc})-\lpk(\pi^{rc})}=\\
t^{\lpk(\pi^{rc})+\val(\pi^{rc})}(1+y)^{\lpk(\pi^{rc})+\val(\pi^{rc})}(y+t)^{\lpk(\pi^{rc})-\val(\pi^{rc})}(1+yt)^{1+\val(\pi^{rc})-\lpk(\pi^{rc})},
\end{multline*}
so in fact the contribution from the two cases are the same.

The remaining contribution comes from the double descents and (remaining)
double ascents of $\acute{\pi}^{rc}$. Note that the number of double
descents of $\acute{\pi}^{rc}$ is equal to $\des(\pi^{rc})-\lpk(\pi^{rc})$.
It is a bit tricker to see that the number of remaining double ascents
of $\acute{\pi}^{rc}$ is equal to $n-1-\val(\pi^{rc})-\des(\pi^{rc})$.
We know that the total number of double ascents of $\acute{\pi}^{rc}$
is equal to $n-\lpk(\pi^{rc})-\des(\pi^{rc})$. If $n$ is a valley
of $\acute{\pi}^{rc}$, then $\lpk(\pi^{rc})=\val(\pi^{rc})+1$, so
the number of remaining double ascents of $\acute{\pi}^{rc}$ is equal
to 
\begin{align*}
n-\lpk(\pi^{rc})-\des(\pi^{rc}) & =n-(\val(\pi^{rc})+1)-\des(\pi^{rc})\\
 & =n-1-\val(\pi^{rc})-\des(\pi^{rc}).
\end{align*}
Otherwise, if $n$ is a double ascent of $\acute{\pi}^{rc}$, then
$\lpk(\pi^{rc})=\val(\pi^{rc})$, so the remaining number of double
ascents of $\acute{\pi}^{rc}$ is again equal to 
\[
(n-\lpk(\pi^{rc})-\des(\pi^{rc}))-1=n-1-\val(\pi^{rc})-\des(\pi^{rc}).
\]
Thus, the contribution from the double descents and (remaining) double
ascents is 
\[
(y+t^{2})^{\des(\pi^{rc})-\lpk(\pi^{rc})}(1+yt^{2})^{n-1-\val(\pi^{rc})-\des(\pi^{rc})}.
\]

Therefore, 
\begin{align*}
\sum_{\sigma\in\mathfrak{B}(\pi)}y^{\neg(\sigma)}t^{\fdes(\sigma)} & =d_{1}(\pi)\cdots d_{n}(\pi)\\
 & =(t(1+y))^{\lpk(\pi^{rc})+\val(\pi^{rc})}(y+t)^{\lpk(\pi^{rc})-\val(\pi^{rc})}(1+yt)^{1+\val(\pi^{rc})-\lpk(\pi^{rc})}\\
 & \qquad\qquad\qquad\qquad\quad\times(y+t^{2})^{\des(\pi^{rc})-\lpk(\pi^{rc})}(1+yt^{2})^{n-1-\val(\pi^{rc})-\des(\pi^{rc})}\\
 & =(1+yt)(1+yt^{2})^{n-1}\left(\frac{t(1+y)(y+t)}{(y+t^{2})(1+yt)}\right)^{\lpk(\pi^{rc})}\\
 & \qquad\qquad\qquad\qquad\qquad\quad\times\left(\frac{t(1+y)(1+yt)}{(1+yt^{2})(y+t)}\right)^{\val(\pi^{rc})}\left(\frac{y+t^{2}}{1+yt^{2}}\right)^{\des(\pi^{rc})},
\end{align*}
and summing over all $\pi\in\Pi$ yields the result.
\end{proof}
By setting $y=1$, we obtain the following specialization relating
the polynomials 
\[
F(\Pi;t)\coloneqq\sum_{\pi\in\mathfrak{B}(\Pi)}t^{\fdes(\pi)}
\]
and

\[
P^{\udr}(\Pi;t)\coloneqq\sum_{\pi\in\Pi}t^{\udr(\pi)},
\]
thus generalizing Corollary \ref{c-fudr}.
\begin{cor}
\label{c-paudr} Let $\Pi\subseteq\mathfrak{S}_{n}$ for $n\geq1$.
Then 
\[
F(\Pi;t)=\frac{(1+t)(1+t^{2})^{n}}{2t}P^{\udr}\left(\Pi^{rc};\frac{2t}{1+t^{2}}\right).
\]
Equivalently, 
\[
P^{\udr}(\Pi;t)=\frac{2v}{(1+v)(1+v^{2})^{n}}F(\Pi^{rc};v)
\]
where $v=\frac{1-\sqrt{1-t^{2}}}{t}$.
\end{cor}
Finally, we give a refinement of Theorem \ref{t-palpkvaldes} that
keeps track of a statistic which is constant on any orbit of the action,
i.e., a statistic on $\mathfrak{S}$. For such a statistic $\st$,
let

\[
F^{\st}(\Pi;y,t,w)\coloneqq\sum_{\pi\in\mathfrak{B}(\Pi)}y^{\neg(\pi)}t^{\fdes(\pi)}w^{\st(\pi)}
\]
and 
\[
P^{(\lpk,\val,\des,\st)}(\Pi;y,z,t,w)\coloneqq\sum_{\pi\in\Pi}y^{\lpk(\pi)}z^{\val(\pi)}t^{\des(\pi)}w^{\st(\pi)}.
\]

\begin{cor}
Let $\st$ be a permutation statistic on $\mathfrak{S}$. Then 
\begin{multline*}
F^{\st}(\Pi;y,t,w)=(1+yt)(1+yt^{2})^{n-1}\\
\times P^{(\lpk,\val,\des,\st)}\left(\Pi^{rc};\frac{t(1+y)(y+t)}{(y+t^{2})(1+yt)},\frac{t(1+y)(1+yt)}{(1+yt^{2})(y+t)},\frac{y+t^{2}}{1+yt^{2}},w\right).
\end{multline*}
\end{cor}

\appendix

\section{Tables of statistics}

Tables 1, 2, 3, 4 summarize the statistics that appear in this paper.
For each statistic, we list the symbol used for the statistic, the
name of the statistic, the (sub)section in this paper where the statistic
is defined, and references to new results in this paper which involve
the statistic.\footnote{We note that every result in this paper for a joint statistic involving
the inversion number $\inv$ gives rise to an analogous result for
the joint statistic obtained by replacing $\inv$ with the inverse
major index $\imaj$, even though we do not explicitly state these
statistics below. See Subsection 4.4 for more details.}

\renewcommand{\arraystretch}{1.4}
\noindent \begin{center}
\begin{longtable}{|c|>{\centering}p{1.5in}|c|>{\centering}p{2in}|}
\caption{Permutation statistics}
\tabularnewline
\endfirsthead
\hline 
Statistic & Name of Statistic & Definition & Results\tabularnewline
\hline 
$\des$ & descent number & \S1 & Theorems \ref{t-bnan}, \ref{t-fnan}, \ref{t-pkdes}, \ref{t-lpkdes},
\ref{t-udr}, \ref{t-lpkvaldes}, \ref{t-pkdesmfs}; 

Corollaries \ref{c-bnan}, \ref{c-pkdes231}, \ref{c-pkdes2ss}\tabularnewline
\hline 
$\pk$ & peak number & \S1 & \tabularnewline
\hline 
$\lpk$ & left peak number & \S1 & Corollary \ref{c-palpk}\tabularnewline
\hline 
$\br$ & number of biruns & \S1 & \tabularnewline
\hline 
$\Des$ & descent set & \S2.1 & \tabularnewline
\hline 
$\val$ & valley number & \S2.1 & \tabularnewline
\hline 
$\udr$ & number of up-down runs & \S2.1 & Theorem \ref{t-udr}; 

Corollaries \ref{c-fudr}, \ref{c-paudr}\tabularnewline
\hline 
$\inv$ & inversion number & \S2.1 & \tabularnewline
\hline 
$(\pk,\des)$ &  &  & Theorems \ref{t-pkdes}, \ref{t-pkdesmfs}, \ref{t-pkdes231f}; Corollaries
\ref{c-pkdes231}, \ref{c-pkdes2ss}\tabularnewline
\hline 
$(\inv,\pk)$ &  &  & Corollary \ref{c-pkq}\tabularnewline
\hline 
$(\inv,\pk,\des)$ &  &  & Theorem \ref{t-pkdesq}\tabularnewline
\hline 
$(\lpk,\des)$ &  &  & Theorems \ref{t-lpkdes}, \ref{t-lpkdesb}, \ref{t-palpkdes}\tabularnewline
\hline 
$(\inv,\lpk,\des)$ &  &  & Theorem \ref{t-lpkdesq}\tabularnewline
\hline 
$(\inv,\lpk)$ &  &  & Corollary \ref{c-lpkq}\tabularnewline
\hline 
$(\inv,\udr)$ &  &  & Theorem \ref{t-udrq}\tabularnewline
\hline 
$(\lpk,\val,\des)$ &  &  & Theorems \ref{t-lpkvaldes}, \ref{t-lpkvaldesb}, \ref{t-palpkvaldes}\tabularnewline
\hline 
$(\inv,\lpk,\val,\des)$ &  &  & Theorem \ref{t-lpkvaldesq}\tabularnewline
\hline 
$\imaj$ & inverse major index & \S4.4 & \tabularnewline
\hline 
$\altdes$ & alternating descent number & \S4.4 & \tabularnewline
\hline 
\end{longtable}
\par\end{center}

\noindent \begin{center}
\begin{longtable}{|c|c|c|c|}
\caption{Type B permutation statistics}
\tabularnewline
\endfirsthead
\hline 
Statistic & Name of Statistic & Definition & Results\tabularnewline
\hline 
$\des_{B}$ & (type B) descent number & \S2.3 & Corollaries \ref{c-bnan}, \ref{c-fnbn}, \ref{c-palpk}\tabularnewline
\hline 
$\fdes$ & flag descent number & \S2.3 & Corollaries \ref{c-fnbn}, \ref{c-fudr}, \ref{c-paudr}\tabularnewline
\hline 
$\neg$ & number of negative letters & \S2.3 & \tabularnewline
\hline 
$(\neg,\des_{B})$ &  &  & Theorems \ref{t-bnegf}, \ref{t-bnan}, \ref{t-fnbn}, \ref{t-lpkdesb},
\ref{t-palpkdes}\tabularnewline
\hline 
$(\neg,\fdes)$ &  &  & Theorems \ref{t-fnegf}, \ref{t-fnan}, \ref{t-fnbn}, \ref{t-lpkvaldesb},
\ref{t-palpkvaldes}\tabularnewline
\hline 
\end{longtable}
\par\end{center}

\noindent \begin{center}
\begin{longtable}{|c|c|c|c|}
\caption{Binary tree statistics}
\tabularnewline
\endfirsthead
\hline 
Statistic & Name of Statistic & Definition & Results\tabularnewline
\hline 
$\nlc$ & number of nodes with no left children & \S5.2 & Corollary \ref{c-tcnlc}\tabularnewline
\hline 
$\tc$ & number of nodes with two children & \S5.2 & \tabularnewline
\hline 
$(\tc,\nlc)$ &  &  & Corollary \ref{c-tcnlc}\tabularnewline
\hline 
\end{longtable}
\par\end{center}

\noindent \begin{center}
\begin{longtable}{|c|c|c|c|}
\caption{Dyck path statistics}
\tabularnewline
\endfirsthead
\hline 
Statistic & Name of Statistic & Definition & Results\tabularnewline
\hline 
$\pk$ & number of peaks (occurrences of $UD$) & \S5.2 & Corollary \ref{c-hkpk}\tabularnewline
\hline 
$\hk$ & number of hooks (occurrences of $DDU$) & \S5.2 & \tabularnewline
\hline 
$(\hk,\pk)$ &  &  & Corollary \ref{c-hkpk}\tabularnewline
\hline 
\end{longtable}
\par\end{center}

\noindent \textbf{Acknowledgements.} The author thanks Ira Gessel
for suggesting this project, reading earlier versions of the manuscript,
and offering many helpful suggestions which greatly improved the quality
of this work; Kyle Petersen for providing feedback on several identities
presented in this paper; and an anonymous referee for their constructive
comments and suggestions.

\phantomsection

\bibliographystyle{plain}
\addcontentsline{toc}{section}{\refname}\bibliography{bibliography}

\end{document}